\documentclass[11pt,letterpaper]{amsart}

\usepackage{color}

\usepackage{graphicx, epstopdf}
\usepackage{cases}

\usepackage{enumerate}
\usepackage{bbm}
\usepackage{amssymb}
\usepackage{amscd}
\usepackage{amsfonts,latexsym,amsmath,amsxtra,mathdots,amssymb,latexsym,mathtools}
\usepackage{mathabx}
\usepackage[all,cmtip]{xy}
\usepackage{geometry}

\usepackage{colordvi}
\usepackage{multicol}
\usepackage{hyperref,dirtytalk}
\usepackage{tikz-cd}
\usepackage{float}
\usepackage{setspace, floatflt}
\usepackage{tensor}
\usepackage[normalem]{ulem}

\allowdisplaybreaks

\newcommand{\BA}{{\mathbb {A}}} \newcommand{\BC}{{\mathbb {C}}} 
  \newcommand{\BH}{{\mathbb {H}}}
  
 \newcommand{\BN}{{\mathbb {N}}} 
 \newcommand{\BR}{{\mathbb {R}}} 
  
 \newcommand{\BZ}{{\mathbb {Z}}}

\renewcommand{\CD}{{\mathcal {D}}} 
 \newcommand{\CF}{{\mathcal {F}}}
\newcommand{\CI}{{\mathcal {I}}} 
\newcommand{\CM}{{\mathcal {M}}} \newcommand{\CP}{{\mathcal {P}}}

\newcommand{\GL}{{\mathrm {GL}}} 
 
\newcommand{\SO}{{\mathrm{SO}}}

 \newcommand{\N}{\mathrm{N}}

\newcommand{\re}{\mathfrak{R}}
\newcommand{\sm}{\mathcal{C}^{\infty}}

\newcommand{\Hom}{\mathrm{Hom}}
  
\newcommand{\Ad}{\mathrm{Ad}}
\newcommand{\triv}{{\mathbf{1}}}
\newcommand{\JL}{{\mathrm{JL}}}

\newcommand{\trd}{{\mathrm{Trd}}}
\newcommand{\nrd}{{\mathrm{Nrd}}}
\renewcommand{\a}{\alpha}
\renewcommand{\b}{\beta}
\renewcommand{\l}{\lambda}
\newcommand{\z}{\zeta}
\newcommand{\x}{\xi}
\DeclareMathOperator{\Sym}{Sym}

\newcommand{\fa}{\mathfrak{a}}

\newcommand{\fg}{\mathfrak{g}}

\newcommand{\fm}{\mathfrak{m}}
\newcommand{\fn}{\mathfrak{n}}

\newcommand{\fp}{\mathfrak{p}}

\newcommand{\fr}{\mathfrak{r}}
\newcommand{\fs}{\mathfrak{s}}

\newcommand{\fz}{\mathfrak{z}}

\newcommand{\fT}{\mathfrak{T}}

\def\sl{\mathfrak{sl}}
\def\-{^{-1}}

\def\D{\Delta}
\def\1{\mathbf{1}}

\def\d{\delta}
\def\e{\epsilon}
\def\k{\kappa}
\def\St{\mathrm{St}}

\def\diag{\mathrm{diag}}

\renewcommand{\Re}{{\mathrm{Re}\,}}

\makeatletter
\g@addto@macro\normalsize{\setlength\abovedisplayskip{3pt}}
\makeatother

\makeatletter
\g@addto@macro\normalsize{\setlength\belowdisplayskip{3pt}}
\makeatother

\newcommand{\delete}[1]{}

\theoremstyle{plain}

\newtheorem{thm}{Theorem}[section] \newtheorem{cor}[thm]{Corollary}
\newtheorem{lem}[thm]{Lemma}  \newtheorem{prop}[thm]{Proposition}

\newtheorem{df}[thm]{Definition}
\newtheorem{fct}[thm]{Fact}

\newtheorem {rem}[thm]{Remark}

\numberwithin{equation}{section}

\newcommand{\sg}{\mathrm{sgn}}

\begin{document}

	\title{The sign of linear periods} 
	
	\author{U.K. Anandavardhanan}
	\address{U.K. Anandavardhanan. Department of Mathematics, Indian Institute of Technology Bombay, Powai, Mumbai, 400076, India}
	\email{anand@math.iitb.ac.in}
	
	\author{H. Lu}
	\address{Hengfei Lu. School of Mathematical Sciences, Beihang University, 9 Nanshan Street, Shahe Higher Education Park, Changping, Beijing, 102206, China}
	\email{luhengfei@buaa.edu.cn}
	
	\author{N. Matringe}
	\address{Nadir Matringe. Institute of Mathematical Sciences, NYU Shanghai, 3663 Zhongshan Road North Shanghai, 200062, China and
		Institut de Math\'ematiques de Jussieu-Paris Rive Gauche, Universit\'e Paris Cit\'e, 75205, Paris, France}
	\email{nrm6864@nyu.edu and matringe@img-prg.fr}
	
	\author{V. S\'echerre}
	\address{Vincent S\'echerre. Laboratoire de Math\'ematiques de Versailles, UVSQ, CNRS, Universit\'e Paris-Saclay, 78035, Versailles, France}
	\email{vincent.secherre@uvsq.fr}
	
	\author{C. Yang}
	\address{Chang Yang. Key Laboratory of High Performance Computing and Stochastic Information Processing (HPC- SIP), School of Mathematics and Statistics, Hunan Normal University, Changsha, 410081, China}
	\email{cyang@hunnu.edu.cn}
	
	\address{Miyu Suzuki. Department of Mathematics, Kyoto University, Kitashirakawa Oiwake-cho, Sakyo-ku, Kyoto 606-8502, Japan}
	\email{suzuki.miyu.4c@kyoto-u.ac.jp}
	
	\address{Hiroyoshi Tamori. Department of Mathematical Sciences, Shibaura Institute of Technology, 307 Fukasaku, Minuma-ku, Saitama City, Saitama, 337-8570, Japan}
	\email{tamori@shibaura-it.ac.jp}
	
	\maketitle
	
	\begin{center}\textit{With an appendix by Miyu Suzuki and Hiroyoshi Tamori}\end{center}

\begin{abstract} 
Let $G$ be a group with subgroup $H$, and let $(\pi,V)$ be a complex representation of $G$. The natural action of the normalizer $N$ of $H$ in $G$ on the space $\Hom_H(\pi,\BC)$ of $H$-invariant linear forms on $V$, provides a representation $\chi_{\pi}$ of $N$ trivial on $H$, which is a character when $\Hom_H(\pi,\BC)$ is one dimensional. If moreover $G$ is a reductive group over a local field, and $\pi$ is smooth irreducible, it is an interesting problem to express $\chi_{\pi}$ in terms of the possibly conjectural Langlands parameter $\phi_\pi$ of $\pi$. In this paper we consider the following situation: $G=\GL_m(D)$ for $D$ a central division algebra of dimension $d^2$ over a local field $F$ of characteristic zero, $H$ is the centralizer of a non central element $\d\in G$ such that $\d^2$ is in the center of $G$, and $\pi$ has generic Jacquet-Langlands transfer to the split form $\GL_{md}(F)$ of $G$. In this setting the space $\Hom_H(\pi,\BC)$ is at most one dimensional. When $\Hom_H(\pi,\BC)\simeq \BC$ and $H\neq N$, we prove that the value of $\chi_{\pi}$ on the non trivial class of $\frac{N}{H}$ is $(-1)^m\e(\phi_\pi)$ where $\e(\phi_\pi)$ is the root number of $\phi_{\pi}$. Along the way we extend many useful multiplicity one results for linear and Shalika models to the case of non split $G$. When $F$ is $p$-adic we also classify standard modules with linear periods and Shalika models, which are new results even when $D$ is equal to $F$.
\end{abstract}

\section{Introduction}

Let $G$ be a group and $H$ be a subgroup of $G$, abbreviated as $H\leq G$. Denote by $N$ the normalizer of $H$ in $G$ and let $\pi$ be a complex representation of $G$.  The group $N$ acts naturally on $\Hom_H(\pi,\BC)$ by $n.L=L\circ \pi(n)^{-1}$. When moreover $\Hom_H(\pi,\BC)$ is one dimensional, this action is given 
by a character $\chi_\pi$ of $N$ which factors through the quotient $\frac{N}{H}$. Furthermore, if $G$ is a reductive group over a $p$-adic field, and $\pi$ is smooth irreducible, it is an interesting problem to determine $\chi_\pi$ in terms of the possibly conjectural Langlands parameter $\phi_{\pi}$ of $\pi$. This paper considers this question in the following situation: $G=\GL_m(D)$ for $D$ a central division algebra of dimension $d^2$ over a local field $F$ of characteristic zero, the group $H$ is the centralizer of a non central element $\d\in G$ such that $\d^2$ is in the center of $G$, and $\pi$ has generic transfer to $\GL_{md}(F)$. Such a situation is indeed a multiplicity one situation, i.e. $\Hom_H(\pi,\BC)$ has dimension at most one, see Theorem \ref{thm mult 1} of this paper for a more general statement. Our main result is the following (see Theorem \ref{thm gen}, and Section \ref{sec LJL} for the unexplained terminology used in its statement):

\begin{thm}\label{thm main}
Suppose that $N\neq H$ so the group $\frac{N}{H}=\{\overline{I_m},\overline{u}\}$ has order two. Let $\pi$ be a smooth irreducible representation of $\GL_{m}(D)$ with generic transfer to $\GL_{md}(F)$, and such that $\Hom_H(\pi,\BC)\neq \{0\}$. Then:
\begin{enumerate}
\item The Langlands parameter $\phi_\pi$ of $\pi$ is symplectic.
\item $\Hom_H(\pi,\BC)\simeq \BC$.
\item The character $\chi_{\pi}$ of $\frac{N}{H}$ given by its natural action on $\Hom_H(\pi,\BC)$ is expressed by the formula \[\chi_\pi(\overline{u})=(-1)^m\e(\phi_\pi),\] where the root number $\e(\phi_\pi)$ is independent of the choice of any non trivial additive character of $F$ used to define it.
\end{enumerate}
\end{thm}

When $F$ is $p$-adic, this in particular extends \cite[Theorem 4]{P} of Prasad which is the special case where $G$ is an inner form of $\GL_2$, and \cite[Theorem 3.2]{LM} which is the special case of split $G=\GL_n(F)$ and $H\simeq \GL_{n/2}(F)\times \GL_{n/2}(F)$. Actually \cite[Theorem 4]{P} was used in \cite{CCL} to study Heegner points on a general class of modular curves, and it was mentioned to us by Li Cai that our main result could have possible number theoretic applications to more general problems of this type. Note also that we only consider the case of distinction by the trivial character of $H$, but considering a quadratic twisting character should make essentially no difference in our method, as was observed and mentioned to us by Binyong Sun. 

The sign of all irreducible representations seems more difficult to compute. Indeed, we don't know if multiplicity at most one holds for standard modules. For this reason it seems that without proving such a result, determining the sign of irreducible representations using the techniques of the present paper would require a detailed study of intertwining periods as in \cite{MatJFA}. We hope to come back to it later. 

When we restrict to representations with generic transfer, then there is an explicit classification of distinguished representations in terms of their Langlands parameters when $\delta^2$ is not a square in the center $Z$ of $G$ and the residual characteristic of $F$ is not $2$: the Prasad and Takloo-Bighash conjecture for trivial twisting character is indeed proved in the papers \cite{X}, \cite{SuzJNT}, \cite{SPTB}, and the final step \cite{SX} which relies on the result of \cite{Ch} as a separate case. Its Archimedean analogue is proved in \cite{ST}. Actually in this paper, we only need to know that the Langlands parameter of such a representation is symplectic, and this follows from \cite{ST}, \cite{X}, \cite{SuzJNT}, \cite{SX} and \cite{Ch} without any restriction on $p$. For $p$-adic $F$ we also establish such a classification (and actually more) when $\d^2$ is a square in $Z$, extending results from \cite{MatCRELLE} and \cite{MatJNT} for the split case, thanks in particular to results of \cite{BW} on discrete series representations. On the other hand for Archimedean $F$, the classification of standard modules with a linear period is provided by Appendix \ref{app ST} by Suzuki and Tamori. In the $p$-adic case the computation of the sign relies on a local/global argument for distinguished cuspidal representations, relying on some global results from \cite{XZ}. It is then extended to distinguished discrete series representations, and then to distinguished generic representations using some explicit linear forms on induced representations called intertwining periods. The use of such linear forms to analyze the sign of induced representations already appears in \cite[Lemma 3.5]{LM}. In fact these linear forms also appear in the local/global argument in the computation the sign of generic unramified representations when 
$\d^2$ is not a square in $Z$, whereas the results of \cite{FJ} are used when it is. In the Archimedean case the sign of distinguished discrete series representations is computed by a direct computation, following the argument of \cite[Theorem 4]{P}. The reduction from distinguished generic series representations to distinguished distinguished series representations is the same as in the $p$-adic case. 

We now briefly describe the different parts of the paper. In Section \ref{sec standard}, for $p$-adic $F$ and $H$ such that $N\neq H$, we recall the classification of $H$-distinguished standard modules in terms of distinguished discrete series representations of $G$ due to Suzuki in the setting of twisted linear periods, i.e. when $\d^2$ is not a square in $Z$ (see \cite{SuzJNT}), and prove it in the setting of linear periods, i.e. when $\d^2$ is a square in $Z$, in Theorem \ref{thm linear std} (see Section \ref{sec lin per} for more on the terminology). We also obtain analogues of Theorem \ref{thm linear std} in \ref{thm chi linear std}, where linear periods are allowed to be twisted by a character, and extend the classification of generic distinguished representations obtained in \cite{MatCRELLE} in Theorem \ref{thm chi linear gen}. Still for $p$-adic $F$, using the Friedberg-Jacquet integrals theory, which is valid in our setting, we are also able to classify standard modules with a Shalika model in Theorem \ref{thm class std shal}, and prove a multiplicity at most one for Shalika models of irreducible representation in Corollary \ref{cor shal mult 1}. Parts of the results described above require extending some results from \cite{MatJNT} to inner forms, which require in particular Appendix \ref{app BZ der}. In Section \ref{sec sign and induction}, we reduce for $p$-adic $F$ the computation of the sign of a representation of $G$ with generic transfer to that of cuspidal representations, using closed and open local intertwining periods. The sign of cuspidal representations is then computed in Section \ref{sec cusp} by a global argument, using local results from \cite{FJ}, global results from \cite{XZ}, and an explicit unramified computation. It came to our attention that one part of our argument is very reminiscent of the proof of \cite[Theorem A]{Pra-Ram}. The main result of the paper is Theorem \ref{thm gen} stated in Section \ref{sec main}. Section \ref{sec Arch} deals with Archimedean local fields. 
There are four appendices. The first one, Appendix \ref{app mult 1}, extends a multiplicity at most one result of Chong Zhang (\cite{Zlinear}) for linear periods of local inner forms of $\GL_n$ from irreducible unitary to all irreducible representations, and proves more results of this type. The second one, Appendix \ref{app BZ der}, extends parts of the theory of Bernstein-Zelevinsky derivatives to representations of non-Archimedean inner forms of $\GL_n$. It is in particular used to compute the sign of Steinberg representations. It is also used in the third appendix, namely Appendix \ref{app mult zero}, to prove that generic representations (and in fact any representation which can be written as a commutative Bernstein-Zelevinsky product of discrete series representations) cannot have a linear model with respect to 
$\GL_l(D)\times \GL_{l'}(D)$ when $|l-l'|\geq 2$, thus extending \cite[Theorem 3.2]{MatJNT}. The fourth appendix, Appendix \ref{app ST} by Miyu Suzuki and Hiroyoshi Tamori, provides a necessary condition for a standard module of $G_m$ to support a linear period when $F$ is Archimedean.\\

\textbf{Acknowledgement.} This paper was written under the impulse of a question asked to the fourth named author by Erez Lapid, and we thank him for this as well as for useful comments, especially for drawing our attention to the paper \cite{LM}. We also thank Binyong Sun, Dipendra Prasad, Li Cai and Miyu Suzuki for useful remarks. Finally the third name author would like to thank the IAMS of Zhejiang University, especially Rui Chen and Yifeng Liu, for inviting him to give a talk about this paper as well as for its warm hospitality. The addition of the Archimedean section of the paper was essentially done there. Hengfei Lu and Chang Yang are supported by the National Natural Science Foundation of China (No 12301031 and 12001191 respectively). 
	
\section{Notations and preliminary results}
	
Let $F$ be a field of characteristic different from $2$, and let $D$ be a division algebra with center $F$ and of dimension $d^2$ over $F$. For $m$ a non-negative integer, we set 
$G_m=\GL_m$. The group $G_0(D)$ is trivial by definition. For $m\geq 1$, we denote by $\delta$ a non central element in $G:=G_m(D)$ such that $\k:=\d^2$ is in the center of $G$, and set $H$ to be the centralizer of $\delta$ in $G$. The group $G$ is equal to the group consisting of invertible elements of the central simple $F$-algebra $A:=\CM_m(D)=:\CM_m$. Then $H$ is a symmetric subgroup obtained as the fixed points of the inner automorphism 
\[\theta:=\Ad(\delta):g\mapsto \delta g \delta^{-1}\] of $G$. We identify the center $Z$ of $G$ with $F^\times$. 

\subsection{Symmetric pairs}

Here we identify three different situations for the pair $(G,H)$. The assertions below follow from basic results in linear algebra, extended to vector spaces over division algebras as in \cite{Lip}. 

\begin{enumerate}[1)]
\item If $\d^2\in (F^\times)^2$, then there are non negative integers $p\geq q$ such that $p+q=m$, and such that $\d$ is up to scaling in $F^\times$ conjugate of \[\d_{L_{p,q}}:=\diag(I_p,-I_q).\] With such a choice of $\delta$ we identify 
$H$ to the maximal block 
diagonal Levi subgroup 
\[L_{p,q}:=\diag(G_p(D),G_q(D))\] of $G$. 
When convenient, we will prefer to 
take \[\delta=\d_{H_{p,q}}:=\diag(I_{p-q},1,-1,\dots,(-1)^{q-1})\] in the conjugacy class of $\delta$, the centralizer of which we denote by $H_{p,q}$ and can be described as in \cite{MatCRAS}. Finally when convenient, for $p=q$, we will take 
\[\delta=\d_{M_{p,p}}:= \begin{pmatrix}  & I_p \\ I_p & \end{pmatrix},\] with $M_{p,p}=H$ the centralizer of $\d_{M_{p,p}}$. 
\item If $\d^2\in F^\times-(F^\times)^2$, then $E=F[\delta]\subseteq A$ is a quadratic extension of $F$, and $n:=md$ is even.
\begin{enumerate}
\item There exists an $F$-algebra homomorphism from $E$ to $D$ (it is unique up to $D^\times$-conjugacy by the Skolem-Noether theorem). When $F$ is a local field, this simply means that $d$ is even. We may assume $E\subseteq D$. Then $H$ identifies with 
$G_m(C_E)$, where $C_E$ is the central $E$-division algebra centralizing $E$ in $D$, and $\d$ identifies with the matrix 
$\d.I_m\in G_m(E)$.
\item There is no $F$-algebra homomorphism from $E$ to $D$. When $F$ is a local field this means $d$ odd. Then we put $D_E:=D\otimes_F E$, and it is a central $E$-division algebra. We realize $D_E$ inside 
$A$ as the matrices set of matrices $\begin{pmatrix} x& \k y\\ y& x\end{pmatrix}$ with $x$ and $y$ in $D$. In other words  
\[\delta=\diag\left(\begin{pmatrix} & \k\\ 1& \end{pmatrix},\dots,\begin{pmatrix} & \k\\ 1& \end{pmatrix}\right)\in G,\] and 
\[H_m:=G_{m/2}(D_E).\]
However we will also take \[\delta=\begin{pmatrix} & \k I_{m/2}\\ I_{m/2}& \end{pmatrix} \] when convenient, in which case 
\[H=H'_m:=\left\{\begin{pmatrix} A& \k B\\ B & A \end{pmatrix}\in G, \ A, \ B \ \in \CM_m(D)\right\}.\]
\end{enumerate}
\end{enumerate}

\subsection{Normalizers}\label{sec normalizer}

In this paragraph we compute the normalizer of $H$ in $G$ in the different cases singled out in the previous paragraph.

\begin{enumerate}[1)]
\item Take $H$ under the form $H=L_{p,q}$. If $p\neq q$, the normalizer $N:=N_G(H)$ is equal to $H$. Indeed for $n$ in $N$, the matrix $n\d_{L_{p,q}}n^{-1}$ must still belong to the center of $H$, so it must be of the form $\diag(\lambda I_p, \mu I_q)$. Reasoning on the possible choices of eigenvalues and their multiplicity, we deduce that $\lambda=1$ and $\mu=-1$. Hence $n$ centralizes $\d_{L_{p,q}}$ and so it belongs to $H$. If $p=q$ a similar reasoning implies that 
\[N=H\sqcup u H,\] where 
\[u=u_m:=\begin{pmatrix} 0 & I_{m/2} \\ -I_{m/2} & 0 \end{pmatrix}.\] If we choose $H=H_{m/2,m/2}$ then we take 
\[u=u_m:=\d_{M_{m/2,m/2}},\] whereas if 
we choose $H=M_{m/2,m/2}$ then we take 
\[u=u_m:=\diag(I_{m/2},-I_{m/2}).\]
\item 
\begin{enumerate} 
\item If $d$ is even, then by \cite{Ch}, one can write $D=C_E+\iota C_E$ where $\iota$ is such that \[\theta(\iota )=-\iota .\] In this situation conjugation by $\iota$ on the central $E$-division algebra $C_E$ induces the Galois involution on $E$ because otherwise $\iota $ would centralize $E$, which it does not. Now take $n\in N$. The inner automorphism $\Ad(n)$ of $\CM_m(D)$ is an $F$-linear automorphism of $\CM_m(C_E)$, which stabilizes its center $E$. If it induces the identity on $E$, then $n\in H$ by the definition of $H$, and if not then $\Ad(\iota ^{-1} n )$ induces the identity of $E$, hence $\iota n\in H $. The conclusion here is that 
\[N=H\sqcup u H,\] where $u=u_m:=\iota I_m$.
\item If $d$ is odd, a similar reasonning, using the Skolem-Noether theorem proves that 
\[N=H\sqcup uH,\] where this time 
\[u=u_m:=\diag(1 , -1,\dots ,(-1)^{m-1})\] if $H=H_m$ and 
\[u=u_m:=\diag(I_m,-I_m)\] if $H=H'_m$.
\end{enumerate}
\end{enumerate}


\subsection{Local Langlands correspondence for inner forms of $\GL_n$}\label{sec LJL}

Here $F$ is a local field of characteristic zero. We denote by $|\ |_F$ the normalized absolute value of $F$. When $F$ is non Archimedean, we denote its residual cardinality by $q_F$, and by $\mathfrak{w}_F$ its uniformizer. We denote by $\nrd$ the reduced norm on $G$, and for $g\in G$ we set \[\nu(g)=|\nrd(g)|_F.\]  We only consider smooth admissible complex representations, which we call representations. When $F$ is Archimedean we moreover require them to be Casselman-Wallach representations as in \cite{FJ} for example. When $F$ is $p$-adic, we recall that thanks to \cite{HTllc} and \cite{Hllc}, one can associate to any irreducible representation $\pi$ of $\GL_n(F)$ a semi-simple representation $\phi_\pi$ of dimension $n$ of the Weil-Deligne group $\mathrm{WD}_F$ of $F$, which we call its Langlands parameter. When $F$ is Archimedean, taking the Weil-Deligne group of $F$ to be equal to its Weil group, one can do the same association thanks to Langlands' work \cite{LArch}, and we refer to \cite{KnappLLC} for more details on $\GL_n$. 

We denote by $\JL$ the Jacquet-Langlands correspondence, which is a bijection from the set of isomorphism classes of discrete series representations, also called essentially square integrable representations of 
$G$, to that of $\GL_n(F)$ (see \cite{DKV} and \cite{Badulescu-Renard} when $F$ is Archimedean). In \cite{Bjlu} when $F$ is non Archimedean and \cite{Badulescu-Renard} when $F$ is Archimedean, the Jacquet-Langlands correspondence was extended to an injective map from the Grothendieck group of finite length representations of $G$ to that of $\GL_n(F)$. These references actually define a surjective morphism $\mathrm{LJ}_n$ in the other direction, which has the Jacquet-Langlands correspondence as a section. 

An irreducible representation of $\GL_n(F)$ is called generic if it has a non degenerate Whittaker model, see \cite{JSeuler1}. Using the standard product notation for normalized parabolic induction, it follows from from \cite{Z} and \cite{JrsArch} that 
and irreducible representation of $\GL_n(F)$ is generic if and only if it is a product of discrete series representations. 

\begin{df}
Let $\pi$ be an irreducible representation of $G$. We say that it is generic if it has generic transfer to $\GL_n(F)$, i.e., if it is of the form 
$\mathrm{LJ}_n(\pi_0)$ for $\pi_0$ a generic representation of $\GL_n(F)$, in which case we set \[\JL(\pi):=\pi_0.\] 
We moreover define the Langlands parameter $\phi_\pi$ by 
\[\phi_\pi:=\phi_{\pi_0}.\]
\end{df}

\begin{fct}
By \cite[Proposition 3.4]{Bjlu}, $\pi$ is generic if and only if it can be written as a product of discrete series representations 
\[\pi=\d_1\times \dots \times \d_r\] such that $\JL(\d_1)\times \dots \times \JL(\d_r)$ is generic. If so 
\[\JL(\pi)=\JL(\d_1)\times \dots \times \JL(\d_r).\]
\end{fct}

\begin{df}
We say that a generic representation $\pi$ of $G$ is symplectic if its Langlands parameter $\phi_{\pi}$ is a symplectic representation of 
$\mathrm{WD}_F$, i.e. if it preserves a non degenerate alternate form.
\end{df}

For $p$-adic $F$, we refer to \cite{Z} and \cite{Tglna} for the parametrization of discrete series representations of $G$ in terms of cuspidal segments, as well as for the vocabulary related to these objects. In \cite{Tglna}, a positive integer $s(\rho)|d$ is attached to any cuspidal representation $\rho$ of $G$. It can be described as the smallest non negative real number $s$ such that $\rho\times \nu^{s(\rho)}\rho$ is reducible. For $\rho$ a cuspidal representation of $G$, we set \[\pi_t(\rho):=\nu^{(1-t)s(\rho)/{2}}\rho\times \dots \times \nu_{\rho}^{(t-1)s(\rho)/{2}}\rho.\]
Then $\pi_t(\rho)$ has a unique irreducible quotient, denoted by $\St_t(\rho)$.

If $\rho$ is a cuspidal representation of $G_a(D)$ with transfer \[\JL(\rho)=\St_r(\rho')\] to $\GL_{ad}(F)$ for some $r$ dividing $ad$ and $\rho'$ a cuspidal representation of $\GL_{\frac{ad}{r}}(F)$, then 

\[\JL(\St_t(\rho))=\St_{tr}(\rho').\] In particular, when $\rho=\triv_{D^\times}$ we have $r=d$ and $\rho'=\triv_{F^\times}$, i.e. 
\[\JL(\St_t(\triv_{D^\times}))=\St_{td}(\triv_{F^\times}),\]

Let $F$ any local field of characteristic zero, and $\pi$ be a generic representation of $G$. For $\psi$ a non trivial character of $F$, we denote by $\epsilon(1/2,\phi_\pi,\psi)$ the Langlands root number attached to $\phi_{\pi}$ as in \cite{Tate}, and we set 
\[\epsilon(\pi,\psi):=\epsilon(1/2,\phi_\pi,\psi).\] It is known that this root number is the same as the one defined in \cite{GJ} via the Godement-Jacquet functional equation. If moreover $\pi$ is symplectic, then $\epsilon(1/2,\phi_\pi,\psi)$ is independent of the choice of $\psi$ and we write:

\[\e(\pi):=\e(\pi,\psi).\]

When $F$ is $p$-adic, we will use the following well-known formulae concerning root numbers which can be found in \cite{GJ} and \cite{GR}. We observe that the formulae \eqref{x} and \eqref{y} below are also valid when $F$ is Archimedean. 

\begin{enumerate}
\item If $\rho$ is a cuspidal representation of $G_a$ with transfer $\JL(\rho)=\St_r(\rho')$ to $\GL_{ad}(F)$, 
then \[\e( \St_t(\rho),\psi)=\e( \rho',\psi)^{rt}\] \textbf{except when} $\rho$ is an unramified character $\mu\circ \nrd$ of $G_1=D^\times$, in which case the formula is given by 
\begin{equation}\label{eq st1} \e( \St_t(\mu\circ \nrd),\psi)=(-\mu(\mathfrak{w}_F))^{dt-1}\e(\mu,\psi)^{dt}.\end{equation} In particular when $\St_t(\mu)$ is self-dual, we observe that 
\[\e( \St_t(\eta),\psi)=\e(\eta,\psi)^{dt}\] when $\mu=\eta$ is the unramified quadratic character of $F^\times$, which is the same formula as that in Equation \eqref{eq st1}, whereas when $\mu=\triv_{F^\times}$ we obtain 
\[\e( \St_t(\triv_{D^\times}),\psi)=(-1)^{dt-1}\e(\triv_{F^\times},\psi)^{dt}.\] 
\item \label{x} If $\pi$ is an irreducible representation of $G$ with central character $\omega_{\pi}$,
\[\e( \pi,\psi)\e( \pi^\vee,\psi)=\omega_{\pi}(-1).\] 
\item \label{y} If $\pi_1$ and $\pi_2$ are irreducible representations of $G_{m_1}$ and $G_{m_2}$ respectively, such that $\pi_1\times \pi_2$ is irreducible, then 
\[\e(\pi_1\times \pi_2,\psi)=\e(\pi_1,\psi)\e(\pi_2,\psi).\]
\item If $\pi$ is a generic unramified representation of $\GL_n(F)$ and $\psi$ is unramified as well, then \[\e( \pi,\psi)=1.\]
\end{enumerate}

\subsection{Local linear periods and their sign}\label{sec lin per}

Here $F$ is a local field of characteristic zero.

\begin{df}
Let $\pi$ be an irreducible representation of $G$ and $\chi$ be a character of $H$. We say that $\pi$ is $\chi$-distinguished if 
$\Hom_H(\pi,\chi)\neq \{0\}$. 
Moreover if $\pi$ is distinguished:

\begin{enumerate}[1)]
\item if $\d^2\in (F^\times)^2$, we call an element in $\Hom_{H}(\pi,\chi)-\{0\}$ a $\chi$-linear period.
\item if $\d^2\notin (F^\times)^2$, we call an element in $\Hom_{H}(\pi,\chi)-\{0\}$ a twisted $\chi$-linear period. Furthermore 
for such an invariant linear form $L$:
\begin{enumerate} 
\item if $d$ is odd we say that $L$ is a twisted $\chi$-linear period of odd type.
\item if $d$ is even we say that $L$ is a twisted $\chi$-linear period of even type.
\end{enumerate}
\end{enumerate}
We will mostly consider the case of the trivial character $\chi=\triv_H$ of $H$, in which case we will omit $\chi$ in the terminology above.
\end{df}

If $\pi$ is distinguished, then by \cite{JR}, \cite{G}, \cite{BM}, Appendix \ref{app mult 1} and  Appendix \ref{app mult zero}, we have much information on the dimension of the space 
$\Hom_{H}(\pi,\BC)$. 

\begin{thm}\label{thm mult 1}
Let $\pi$ be an irreducible representation and suppose that $H\not \simeq H_{p,q}$ such that $|p-q|\geq 2$. Then 
\[\dim(\Hom_H(\pi,\BC))\leq 1.\] If $|p-q|\geq 2$ this conclusion holds for irreducible unitary representations as well as for generic representations. In fact for generic representations we have \[\dim(\Hom_H(\pi,\BC))=0\] when $|p-q|\geq 2$.
\end{thm}

\begin{rem}
When $G$ is split, it is proved in \cite{JR} that $\dim(\Hom_{H_{p,q}}(\pi,\BC))\leq 1$ for any irreducible $\pi$, without any restriction on $p$ and $q$. 
\end{rem}

In particular for any distinguished generic representation $\pi$ of $G$, we define the character $\chi_\pi$ of 
$\frac{N}{H}$ by the equality for all $n\in N$:
\[L\circ \pi(n)^{-1}=\chi_{\pi}(\overline{n})L.\]

The problem is that of computing $\chi_\pi$ in terms of its Langlands parameter $\phi_{\pi}$. 

We recall that $\frac{N}{H}$ is trivial in the case of linear periods when $p\neq q$, so we forget about this case unless explicitly stated. Hence our three cases are now:
\begin{enumerate}[1)]
\item Linear periods with respect to $H$ conjugate to $L_{m/2,m/2}$.
\item 
\begin{enumerate} 
\item Twisted linear periods of odd type.
\item Twisted linear periods of even type.
\end{enumerate}
\end{enumerate}

Distinguished generic representations of $G$ always have a symplectic Langlands parameter as explained in Corollary \ref{cor symp gen}. As we shall see that the solution to our problem involves the quantity $\epsilon(\pi)$, it will be convenient to 
define the \textit{sign} of the (possibly twisted linear period on the) distinguished representation $\pi$ by:

\[\sg(\pi)=\frac{\chi_{\pi}(\overline{u})}{\epsilon(\pi)}=\chi_{\pi}(\overline{u}) \epsilon(\pi).\]

\section{Results on distinguished standard modules}\label{sec standard}

\subsection{Friedberg-Jacquet integrals for inner forms and applications to linear periods and Shalika models}\label{sec FJ}

Here $G=G_m(D)=:G_m$ with $m$ even. In Section \ref{sec gen sign}, allowing $F$ to be Archimedean, we will use the Friedberg-Jacquet integrals in the case where $G$ is $F$-split. Indeed the full theory of such integrals is developed in \cite{FJ} over all local fields when $G$ is $F$-split. In the present paragraph we consider general $G$, but restrict to $F$ non Archimedean. We explain why the Friedberg-Jacquet integrals are well defined, and why their gcd is the Godement-Jacquet L-function for $G$ not necessarily $F$-split. The very clever computations in \cite{FJ} easily extend to our $G$, but some proofs of \cite{FJ} are only given for Archimedean $F$ as this case is more difficult. The details of their proofs were written in \cite{MatJNT} and \cite{MatBF2} in the non Archimedean case for split $G$, but the situation in \cite{MatJNT} is complicated by the fact that the Shalika functionals are with respect to the Shalika model of the mirabolic subgroup. Hence some integrals of Shalika functions in \cite{MatJNT} are not obviously absolutely convergent anymore, as in the case of Shalika functionals with respect to the Shalika subgroup of $G$. Here we briefly explain why \cite{FJ} applies without modifications to our setting, giving a simplified version of the arguments in \cite{MatJNT}. The reason for studying Friedberg-Jacquet integrals is that following \cite{MatBF2} for split $G$, one proves using such integrals that if $\delta$ is a discrete series representation of $G$, and if $\chi$ is a character of $H=L_{m/2,m/2}$ trivial on $G_{m/2}$ diagonally embedded, then $\delta$ has a $\chi$-linear period if and only if it has a Shalika model. This result is then used in Section \ref{sec dist stand} for the classification of standard modules of $G_m$ with a linear period when $m$ is odd, though this classification is not really needed in this paper. However it is used in Appendix \ref{app mult 1} to prove multiplicity one of linear periods for the pair $(G_{2q+1},G_{q+1}\times G_{q})$.

First we recall the definition and properties of Godement-Jacquet integrals for induced representations of $G$. We recall that $n=md$. Let
\[\pi=\pi_1\times \dots \times \pi_r\] be a representation of $G$ with each $\pi_i$ irreducible. For $\Phi$ a Schwartz function on $\CM:=\CM_m$, and $c$ a matrix coefficient of $\pi$, we set 
\[Z(s,\Phi,c):=\int_G \Phi(g)c(g)\nu(g)^{s+\frac{(n-1)}{2}} dg .\]
We moreover say that $\pi$ is a standard module of $G$ if each $\d_i:=\pi_i$ is a discrete series representation such that $e(\delta_i)\geq e(\delta_{i+1})$ for $i=1,\dots,r-1$, where $e(\delta)$ is the unique real number such that $\nu^{-e(\delta)} \delta$ is unitary. We recall that the Langlands classification (see \cite{Tglna}) says that standard modules of $G$ have a unique irreducible quotient, and conversely that each irreducible representation of $G$ is the quotient of a standard module.

The first part of the following statement is proved in \cite{GJ}. The second part for Langlands quotients of standard modules is proved in \cite[(2.3) Proposition, (3.4) Theorem]{Jprincipal} for split $G$ with a proof valid for any $G$, following \cite{GJ}. 

\begin{thm}[Jacquet]\label{thm ind J}
\begin{enumerate}
\item\label{J1} For $\pi=\pi_1\times \dots \times \pi_r$ a representation of $G$ with each $\pi_i$ irreducible, the integral $Z(s,\Phi,c)$ converges for $s$ of real part large enough, and extends to a rational function of $q_F^{-s}$. Moreover the vector space spanned by the integrals $Z(s,\Phi,c)$ is a fractional ideal of $\BC[q_F^{\pm s}]$ and the normalized gcd $L(s,\pi)$ of this fractional ideal is equal to the following the product of Godement-Jacquet $L$-factors 
\[L(s,\pi)=\prod_{i=1}^r L(s,\pi_i).\]
\item\label{J2} If moreover $\pi$ is a standard module with Langlands quotient $\mathrm{LQ}(\pi)$, then
\[L(s,\pi)=L(s,\mathrm{LQ}(\pi)).\]
\end{enumerate}
\end{thm}

We now denote by $S_m$ the Shalika subgroup of $G_m$. By definition it is the set of matrices 
\[S_m:=\{s(g,x):=\diag(g,g) \begin{pmatrix} I_{m/2} & x \\ & I_{m/2}  \end{pmatrix}, \ g\in G_{m/2}, x\in \CM_{m/2}\}.\]
We fix $\psi$ a non trivial character of $F$, and denote by $\Theta$ the character of $S_m$ defined by the formula 
\[\Theta(s(g,x))=\psi(\trd(x)),\] where $\trd$ is the reduced trace of $\CM_{m/2}$. If $\pi$ is a smooth representation of $G$, and $L\in \Hom_{S_m}(\pi,\Theta)-\{0\}$, we call the space \[S_L(\pi,\Theta):=\{S_v:g\mapsto L(\pi(g)v), \ v\in \pi\}\] the Shalika model of $\pi$ attached to $L$ (and $\psi$). When $G$ is split it is known thanks to \cite{JR} (and \cite{AGJ} when $F$ is Archimedean) that the space 
$\Hom_{S_m}(\pi,\Theta)$ is at most one dimensional for any irreducible $\pi$. For $G$ as in this paragraph and $\pi$ a discrete series representation, the multiplicity of Shalika functionals is again known to be at most one by \cite{BW}. We will 
extend their result to all irreducible representations in Corollary \ref{cor shal mult 1} and thus remove the dependence to $L$ in the notation for the Shalika model of irreducible representations. The following generalizes the $p$-adic part of \cite[Proposition 3.1]{FJ}, and follows from its proof.

\begin{prop}\label{prop FJ}
Let $\pi$ be a representation of $G$ as in Theorem \ref{thm ind J}, \eqref{J1}, and suppose that \[\Hom_{S_m}(\pi,\Theta)\neq \{0\}.\] Take $L\in \Hom_{S_m}(\pi,\Theta)-\{0\}$ and $S\in S_L(\pi,\Theta)$, and $\alpha$ a character of $F^\times$. Then the integral 
\[\Psi_{\alpha}(s,S)=\int_{G_{m/2}} S(\diag(g,I_n))\alpha(\nrd(g))\nu(g)^{s-1/2}dg\] is absolutely convergent for $s$ of real part large enough and extends to a rational function of $q_F^{-s}$. Moreover the vector space spanned by the integrals $\Psi_{\alpha}(s,S)$ is a fractional ideal of $\BC[q_F^{\pm s}]$ and the normalized gcd of this fractional ideal is the Godement-Jacquet $L$-factor $L(s,\alpha \pi)$. 
\end{prop}
\begin{proof}
We can suppose $\alpha$ trivial as in the proof of \cite[Proposition 3.1]{FJ} and set $\Psi(s,S):=\Psi_{\alpha}(s,S)$. We denote by $K_m$ the maximal compact subgroup 
$\GL_m(O_D)$ of $G_m$ where $O_D$ is the ring of integers of $D$. For $S\in S_L(\pi,\Theta)$ and $\phi\in \sm_c(\CM_{m/2}\times G_{m/2} \times K_m)$ a smooth function with compact support in $\CM_{m/2}\times G_{m/2} \times K_m$, we define 
\[S_{\phi}(g)=\int_{u,b,k}S\left(g\begin{pmatrix} b^{-1} & \\ & 1\end{pmatrix}\begin{pmatrix} 1 & u \\ & 1\end{pmatrix}
\begin{pmatrix} 1 & \\ & b\end{pmatrix}k)\right)\phi(u,b,k)\nu(b)^{n/2}dbdudk,\] with variable 
\[(u,b,k)\in \CM_{m_2}\times G_{m/2}\times K_m.\]
Clearly $S_{\phi}\in  S_L(\pi,\Theta)$ and one can write $S=S_{\phi}$ for some well chosen $\phi$ by \cite[Lemma 4.2]{MatJNT}. 
For such a $\phi$, we denote by $K_{\phi}$ its support in the variable 
$b\in  G_{m/2}$, and by $\phi'$ the characteristic function of $K_{\phi}^{-1}$. We then attach to $\phi$ a map 
\[\Phi_{\phi}\left(\begin{pmatrix} a & x \\ & b \end{pmatrix}, k\right)= \int_{u,v}\phi(u,b,k)\phi'(v)\Theta(ua-vx)dudv\] with compact support in the variable $(a,x,b,k) \in \mathcal{M}_{m/2} \times \mathcal{M}_{m/2} \times G_{m/2} \times K_m$ by properties of the Fourier transform. 

Let $\Omega$ be the open set of matrices $\begin{pmatrix} a & b \\ c & d \end{pmatrix}$ of $\mathcal{M}_m$ with $(c \ d)$ of rank $m/2$ 
in $\mathcal{M}_{m/2,m}$, and $\Omega_0$ its subset of matrices $\begin{pmatrix} a & b \\ 0 & d\end{pmatrix}$ with $d$ invertible. The map 
\[r:(p,k)\mapsto pk\] from $\Omega_0\times K_m$ to $\Omega$ is proper and surjective.
We then define $\Phi_{\phi}*$ on $\Omega$, by 
\[\Phi_{\phi}*(pk)=\int_{k'\in K_m\cap P_{(m/2,m/2)}} \Phi_{\phi}(pk'^{-1},k' k)dk',\] 
for $(p,k)\in \Omega_0\times K_m$. The map $\Phi_{\phi}*$ is well defined and has compact support in $\Omega$ as $r$ is proper, it is moreover fixed by some open compact subgroup of $K_m$ which stabilizes $\Omega$ by right translation, hence it is smooth. We thus extend $\Phi_{\phi}*$ by zero outside $\Omega$ to obtain a smooth function on $\CM_m$. We now define the integrals 
\[I(S,\Phi_{\phi},a,b)=\int_{x\in \mathcal{M}_{m/2},k\in K_m}S\left(\begin{pmatrix} a & x \\ & b \end{pmatrix} k\right)
\Phi_{\phi}\left(\begin{pmatrix} a & x \\ & b \end{pmatrix},k\right)dx dk\] and 
\[J(S,\phi,a,b)=\]
\[\int_{u\in \mathcal{M}_{m/2},k\in K_m}S\left(\begin{pmatrix} a &  \\ & I_{m/2} \end{pmatrix}\begin{pmatrix} I_{m/2} & u \\ & I_{m/2} 
\end{pmatrix}\begin{pmatrix} I_{m/2} &   \\ & b \end{pmatrix} k\right)\phi(u,b,k)du dk.\] It is proved in \cite[Lemma 4.3]{MatJNT} that they both converge absolutely and are equal, and have compact support with respect to the variables 
$a\in \mathcal{M}_{m/2}$, and $b\in G_{m/2}$, and $k\in K_m$. Now we claim that the integral 
\[\int_{(a,b)\in G_{m/2}^2}I(S,\Phi_{\phi},a,b) \nu(a)^{s+n}\nu(b)^{s+\frac{(n-1)}{2}} da db\] converges absolutely for $s$ of real part large enough, and is of the form $P(q^{-s})L(s,\pi)$ for $P\in \BC[X^{\pm 1}]$ in its realm of convergence. Indeed: 
\[\int_{(a,b)\in G_{m/2}^2}|I(S,\Phi_{\phi},a,b)\nu(a)^{s+n}\nu(b)^{s+\frac{(n-1)}{2}}| da db\]
\[\leq 
\int_{(a,b)\in G_{m/2}^2} I(|S|,|\Phi_{\phi}|,a,b) \nu(a)^{\re(s)+n}\nu(b)^{\re(s)+\frac{(n-1)}{2}} da db\]
\[= 
\int_{(a,b)\in G_{m/2}^2}I(|S|,|\Phi_{\phi}|,a,b)|\nu(a)^{\re(s)+n}\nu(b)^{\re(s)+\frac{(n-1)}{2}} da db\]
\[=
\int_{g\in G}|S(g)||\Phi_{\phi*}(g) |\nu(g)^{\re(s)+\frac{(n-1)}{2}} dg\]
\[=\int_{g\in G}|S^U(g) \Phi_{\phi*}(g) |\nu(g)^{\re(s)+\frac{(n-1)}{2}} dg\] 
where the penultimate equality follows from Iwasawa decomposition, and in the last one $U\leq K_m$ is a compact open subgroup of $G$ fixing $\Phi_{\phi*}$ under right translation and $S^U$ is the right average of $S$ by $U$. Because $S^U$ is a genuine matrix coefficient of $\pi$, we recognize the absolute value of an integrand of a Godement-Jacquet integral for $\pi$ in the last equation hence the absolute convergence. This also implies the equality 
\[\int_{(a,b)\in G_{m/2}^2}I(S,\Phi_{\phi},a,b)\nu(a)^{s+n}\nu(b)^{s+\frac{(n-1)}{2}} da db=
Z(s,\Phi_{\phi*},S^U)\] hence the existence of the Laurent polynomial $P$. On the other hand 
\[\int_{(a,b)\in G_{m/2}^2}I(S,\Phi_{\phi},a,b) \nu(a)^{s+n}\nu(b)^{s+\frac{(n-1)}{2}} da db\]
\[=\int_{(a,b)\in G_{m/2}^2}J(S,\phi,a,b) \nu(a)^{s+n}\nu(b)^{s+\frac{(n-1)}{2}} da db\]
\[=\Psi(s,S_{\phi}).\] So far we have proved every part of the statement we want to obtain, except the assertion on the normalized gcd of 
the integrals $\Psi(s,S).$ This assertion is easier and follows word for word from \cite[beginning of p.111 and Lemmas 3.2 and 3.3]{FJ}.
\end{proof}

For $a$ and $b$ two non negative integers, we define the character $\chi_{\alpha}$ of $L_{a,b}$ by 
\[\chi_{\alpha}(\diag(g_1,g_2))=\alpha(\nrd(g_1)\nrd(g_2)^{-1}).\]

Let $V_{\pi}$ be the space of $\pi$ as in Theorem \ref{thm ind J}, \eqref{J1}. For $L\in \Hom_{S_m}(\pi,\Theta)$, and $v\in \pi$, we define 
$S_{v,L}\in S_L(\pi,\Theta)$ by the formula \[S_{v,L}(g)=L(\pi(g)v).\] We then define the linear map 
\[\Lambda_{L,\alpha}:v\mapsto \lim_{s\rightarrow \frac{1}{2}} \frac{\Psi_{\alpha}(s,S_{v,L})}{L(s,\alpha \pi)}\] which makes sense thanks to 
Proposition \ref{prop FJ}. As explained in \cite[p.117]{JR}, we deduce the following from Proposition \ref{prop FJ}.

\begin{cor}\label{cor shal mult 1}
Let $\pi$ be a representation of $G$ as in Theorem \ref{thm ind J}, \eqref{J1}, and suppose moreover that it has a Shalika model, 
then it is $\chi_{\alpha}$-distinguished. More precisely we have an exact sequence 
\[0\to \Hom_{S_m}(\pi,\Theta)\to \Hom_{L_{m/2,m/2}}(\pi,\chi_{\alpha}).\]
In particular \[\dim(\Hom_{S_m}(\pi,\Theta))\leq 1\] whenever $\pi$ is irreducible. 
\end{cor}
\begin{proof}
The linear map \[L\in \Hom_{S_m}(\pi,\Theta)\mapsto \Lambda_{L,\alpha^{-1}}\in \Hom_{L_{m/2,m/2}}(\pi,\chi_{\alpha})\] is injective as $\Lambda_L$ is nonzero as soon as $L$ is thanks to Proposition \ref{prop FJ}. The second statement follows from 
Appendix \ref{app mult 1} which uses the classification of standard modules with a linear period (i.e. $\alpha=\triv_{F^{\times}}$) given in Theorem \ref{thm linear std}. 
\end{proof}

\begin{cor}\label{cor sh vs twisted linear}
Let $\delta$ a discrete series representation of $G$. Then $\delta$ has a $\chi_\alpha$-linear period if and only if it has a Shalika model, and if so it is symplectic, hence selfdual, and $\Hom(\delta,\chi_\alpha)$ has dimension one.
\end{cor}
\begin{proof}
All arguments in \cite[Section 2]{MatBF2}, in particular Proposition 2.2 there, are valid for $G$. We mention in particular that thanks to Proposition \ref{prop FJ}, the statement of \cite[Proposition 2.3]{MatBF2} is valid for the standard module $\pi$ of $G$, hence the proof of \cite[Theorem 2.1]{MatBF2} works without any modification. We highlight that the most delicate part there is the extension of linear form argument at the end of the proof of 
\cite[Theorem 2.1]{MatBF2}, which works equally well here thanks to Theorem \ref{thm ind J}, \eqref{J2}.
\end{proof}

\begin{rem}
In the split case it is proved in \cite[Corollary 3.6]{Yunitary} that Corollary \ref{cor sh vs twisted linear} holds more generally for generic representations, using Gan's technique of relating periods via Theta correspondence. It might extend to the case of inner forms, but this verification is left for later. On the other hand, the unfolding principle of \cite{SV} does not seem to extend easily to linear periods twisted by a non unitary character, as was already observed in \cite[Remak 2.1]{MatBF2} in the split case. In Theorem \ref{thm chi linear std} below we extend \cite[Corollary 3.6]{Yunitary} to standard modules for a wide family of characters $\chi_{\alpha}$.
\end{rem}

\begin{rem}
For $F$-split $G$, the sign of generic representations is given at once by the functional equation of the Friedberg-Jacquet integrals. Again this functional equation is established using the realization of contragredient using transpose inverse, so this is one reason why the theory above does not include this functional equation. Note that our sign computation will prove the Friedberg-Jacquet functional equation of generic unitary representations of $G$ at least at the value $s=1/2$.  
\end{rem}

\subsection{Preliminaries for the geometric lemma applied to linear periods}\label{sec geom lem prelim}

Here we provide the necessary material to apply the geometric lemma of \cite[Corollary 5.2]{OJNT} in the context of linear periods. Let $P=P_{(m_1,\dots,m_r)}$ be a standard upper block triangular parabolic subgroup of $G$ attached to the composition $(m_1,\dots,m_r)$ of $m$, and suppose $l\geq l'$ are two non negative integers such that $l+l'=m$. We recall that from an immediate generalization of the arguments at the beginning of \cite[Section 3]{MatCRELLE} (see also \cite{Yunitary} for a different approach), the double cosets 
$P\backslash G/ H_{l,l'}$ are parametrized by sequences of non negative integers of the form 

\[s=\] 
\[(m_{11}^+,m_{11}^-,m_{12},\dots, m_{1r}, m_{21},m_{22}^+,m_{22}^-,m_{23},\dots,m_{2r},\dots,m_{r1},\dots,m_{rr-1},m_{rr}^+,m_{rr}^-)\] such that if one puts \[m_{ii}:=m_{ii}^+ + m_{ii}^-,\] then the matrix $(m_{ij})_{i,j=1,\dots,r}$ is symmetric, and for all $i$ one has \[\sum_{j=1}^r m_{ij}=m_i,\] and \[\sum_{i=1}^r m_{ii}^+-\sum_{i=1}^r m_{ii}^-=l'-l.\]
We denote by $\CI_P$ this set of parametrizing sequences. Given $s \in \CI_P$, we can associate to it a standard parabolic subgroup $P_s = M_sU_s \subset P$ with the standard Levi $M_s$ of type
\begin{align*}
	(m_{11},\cdots,m_{1r},m_{21}\cdots,m_{rr-1},m_{rr}),
\end{align*} where we ignore the zero's in the above partition of $m$, as in \cite[Section 3]{MatCRELLE}. 
We also have an involution $\theta_s$ on $M_s$ and a positive character $\delta_s$ of $M_s^{\theta_s}$, that was denoted by 
\[\delta_s:=\delta_{P_s^{\theta_s}} \delta_{P_s}^{-1/2}\] and was computed explicitly in \cite[Proposition 3.6]{MatCRELLE}. For $\alpha$ a character of $F^\times$, and $a$ and $b$ two non negative integers, we denote by 
\[\xi_{\alpha}:H_{ll'}\to \BC^\times\] the character of $H_{ll'}$ corresponding via the natural isomorphism between $H_{ll'}$ and $L_{ll'}$ described in \cite[Section 2]{MatJNT} to the character 
\[\chi_{\alpha}:L_{ll'}\to \BC^\times \] defined in Section \ref{sec FJ}. We observe that the computation in \cite[Proposition 3.6]{MatCRELLE} remains valid in our setting, under the following form:

\begin{lem}\label{Proposition 3.6 MatCRELLE}
Set \[\mu:=\nu^d,\] and \[h=\diag(h_{11}^+,h_{11}^-,h_{12},\dots,h_{rr-1}, h_{rr}^+,h_{rr}^-)\in M_s^{\theta_s}.\] Then: 
\begin{enumerate}
\item \[\delta_s(h)=\prod_{1\leq i<j \leq r} \mu(h_{ii}^+)^{\frac{m_{jj}^+-m_{jj}^-}{2}}
\mu(h_{ii}^-)^{\frac{m_{jj}^--m_{jj}^+}{2}}\mu(h_{jj}^+)^{\frac{m_{ii}^+-m_{ii}^-}{2}}\mu(h_{jj}^-)^{\frac{m_{ii}^--m_{ii}^+}{2}}.\]
\item \[\xi_{\alpha}(h)=\prod_{i=1}^r \xi_\alpha(h_{ii}^+)\xi_\alpha(h_{ii}^-)^{-1}.\]
\end{enumerate}
\end{lem} 

\subsection{Distinguished generalized Steinberg representations}\label{sec steinberg}

We start with a result that will be used in classifying standard modules with linear periods in terms of their essentially square-integrable support.

\begin{thm}\label{thm p=q}
Let $\pi$ be a discrete series representation of $G_m$ for $m\geq 2$ and let $l$ and $l$ be two non-negative 
integers such that $l+l'=m$. If $\pi$ is $(H_{l,l'},\mu)$-distinguished for $\mu$ a character of $H_{l,l'}$, then $l=l'$.
\end{thm}
\begin{proof}
We write $\pi=\St_t(\rho)$ for $\rho$ a cuspidal representation of $G_r$. We can suppose that both $t$ and $r\geq 2$ thanks to Propositions \ref{prop cusp p=q} and \ref{prop st p=q}. Then by an explicitation of 
\cite[Corollary 5.2]{OJNT}, we deduce thanks to the case $t=1$ and $r\geq 2$ that $m_{ii}^+=m_{ii}^-$ whenever $m_{ii}$ is nonzero so $l=l'$. We refer to the proof of Theorem \ref{thm linear std} below for more details on such type of computation in a more difficult case.
\end{proof}

We will also use the following fact, which follows from the results in \cite{SX} and \cite{X} in the case of twisted linear periods, and is a special case of Corollary \ref{cor sh vs twisted linear} for linear periods.

\begin{prop}\label{prop symp disc}
If $\pi$ is a discrete series representation of $G$ which is distinguished, then $\pi$ is symplectic, and hence it is selfdual.
\end{prop}

\begin{rem}
In the proof of \cite[Theorem 5.1]{MatJNT}, there is an order mistake: p.16 the 
third equality should read \[\int_{M_{2m}\backslash G_{2m}}|\phi(g)|^2dg=\int_{\CM_m}(|\phi|^2)^K(x) dx.\] Observing that 
$\int_{\CM_m}(|\phi|^2)^K(x) dx$ is also equal to $\int_K(\int_{\CM_m}|\phi(xk)|^2 dx)dk$, the Parseval identity should then be applied to each $\phi_k:=\phi( \ . \ k)$ restricted to $\CM_m$, i.e. equality number four should start as 
\[\int_K(\int_{\CM_m}|\phi(xk)|^2 dx)dk=\int_K(\int_{\CM_m}|\widehat{\phi_k}(x)|^2 dx)dk=...\] The rest of the equalities should then be corrected accordingly, and we refer to the proof given in \cite[Lemme 5.7]{Duhamel} for the correct full proof.
\end{rem}

The following result is enough for our purpose, and in the case of linear periods it follows from the proof of Theorem \ref{thm p=q} except when $r=1$. 

\begin{prop}\label{prop st implies cusp}
Let $\rho\neq \triv_{D^\times}$ be a cuspidal representation of $G_r$, and suppose that $\St_t(\rho)$ is $H$-distinguished. Then there is up to scaling at most one $H$-invariant linear form on 
$\pi_t(\rho)$, so that the $H$-invariant linear form on $\St_t(\rho)$ descends from that on $\pi_t(\rho)$. Moreover if $t$ is odd, then $\rho$ is distinguished. 
\end{prop} 
\begin{proof}
In the case of twisted linear periods, this is part of the proof of \cite[Proposition 5.6]{BM}. In the case of linear periods, 
 when $r\geq 2$, it follows again from the geometric lemma \cite[Corollary 5.2]{OJNT} and \cite[Section 3]{MatCRELLE} that necessarily $m_{i,t+1-i}=r$ for all $i$ smaller than $\lfloor (t+1)/2 \rfloor$, and moreover that $m_{\frac{t+1}{2} \frac{t+1}{2}}^+=m_{\frac{t+1}{2} \frac{t+1}{2}}^-=r/2$ when $t$ is odd. This implies our statement in the case $r\neq 1$. When $r=1$ the assertions for twisted linear periods are proved in \cite{Ch}, and we explain why they also hold for linear periods. By Proposition \ref{prop symp disc}, we can suppose that $\rho\neq \triv_{D^\times}$ is quadratic character $\eta$ of $D^\times$. If some non open orbit contributed to the distinction of $\pi_t(\rho)$, because $\eta$ takes negative value, applying the geometric lemma would lead to a sign issue as in \cite[Proof of Proposition 3.6]{Matringe-Inner+GLn+Steinberg}. Finally note that in this case the integer $t$ cannot be odd for central character reasons.
\end{proof}

\subsection{Classification of distinguished standard modules}\label{sec dist stand}

Here $F$ is a $p$-adic field. The Langlands parameter of an irreducible representation of $G$ is understood in terms of its essentially square-integrable support once it is realized as the quotient of the unique standard module lying above it. Now, it is clear that if the Langlands quotient of a standard module is distinguished, then the standard module is itself distinguished so this observation could be a starting point for computing the sign of all distinguished irreducible representations of $G$. However in this paper  we only compute the sign of generic representations. We feel that the computation of the sign for all irreducible representations would involve much finer properties of admissible intertwining periods on standard modules. Nevertheless, we state the classification of distinguished standard modules in terms of their essentially square-integrable support, from which the classification of distinguished generic representations follows. 

For twisted linear periods, this classification is due to Suzuki \cite[Theorem 1.3]{SuzJNT}, and we recall it in \eqref{eq suz} of Theorem \ref{thm linear std} below. Here we provide the proof in the case of linear periods. We also consider the case $p=q+1$ in order to obtain the 
full statement of Theorem \ref{thm mult 1} later in Appendix \ref{app mult 1}. We follow the main  steps of the proof of \cite[Theorems 3.13 and 3.14]{MatCRELLE}, but we in fact correct the proof in question as one part of Claim (2) in the proof of \cite[Theorems 3.14]{MatCRELLE} is incorrect. The following theorem also holds for Archimedean $F$ thanks to \cite{ST}, Appendix \ref{app ST}, and \cite[Theorem 5.4]{MOY}.  

\begin{thm}\label{thm linear std}
Let $\pi=\d_1\times \dots \times \d_r$ be a standard module of $G$ where 
each $\d_i$ is an essentially square integrable representation of $G_{m_i}$, and for any $1\leq k \leq r$ set $\pi_k$ to be the standard module 
\[\pi_k:=\d_1\times \dots \times \d_{k-1} \times  \d_{k+1} \times \dots \times \d_r.\]
\begin{enumerate}
\item\label{eq suz} The representation $\pi$ has a twisted linear period if and only if there exists an involution $\tau\in S_r$ such that for all $1\leq i \leq r$ one has $\d_{\tau(i)}=\d_i^\vee$, and moreover each $\d_i$ has a twisted linear period when $\tau(i)=i$.
\item\label{b} Suppose that $G=G_m$ for $m$ even, and that $H=H_{m/2,m/2}$. The representation $\pi$ is distinguished if and only if there exists an involution $\tau\in S_r$ such that for all $1\leq i \leq r$ one has $\d_{\tau(i)}=\d_i^\vee$, and moreover $m_i$ is even and $\d_i$ is $H_{m_i/2,m_i/2}$-distinguished when $\tau(i)=i$.
\item\label{c} Suppose that $G=G_m$ for $m$ odd, and that $H=H_{(m+1)/2,(m-1)/2}$. Then $\pi$ is distinguished if and only if there exists an index $1\leq i_0 \leq r$ such that $\d_{i_0}=\triv_{D^\times}$, and $\pi_{i_0}$ is $H=H_{(m-1)/2,(m-1)/2}$-distinguished.
\end{enumerate}
\end{thm}   
\begin{proof}
We only deal with linear periods. We refer for Section \ref{sec geom lem prelim} for the definitions of 
$\mu$ and $\xi_{\alpha}:H_{l,l'}\to \BC^\times$, and we set \[\mu_{\alpha}=\xi_{\alpha}^d.\]

Since the right to left implications of \eqref{b} and \eqref{c} are easier, we start with them. When $m$ is even we can suppose that 
\[\pi=\d_1\times \dots \times \d_a \times  \tau_1 \times \dots \times \tau_b \times  \d_a^\vee \times \dots \times \d_1^\vee\] where the $\tau_i$ are distinguished discrete series representations. The middle product \[\tau_1 \times \dots \times \tau_b\] is distinguished by \cite[Proposition 3.8]{MatCRELLE} so that the full product is distinguished by \cite[Proposition 7.2]{OJNT}. When $m$ is odd we can suppose that 
\[\pi=\d_1\times \dots \times \d_a \times  \tau_1 \times \dots \times \tau_b \times \triv_{D^\times} \times  \d_a^\vee \times \dots \times \d_1^\vee.\] Note that the representations $\tau_i$ are $\xi_{\alpha}$-distinguished for any choice of $\alpha$ thanks to Corollary \ref{cor sh vs twisted linear}, in particular they are $\mu_{|\ |^{\frac{1}{2}}}$-distinguished. Then the representation $ \tau_1 \times \dots \times \tau_b  \times \triv_{D^\times}$ is $H_{k+1,k}$-distinguished for the relevant $k$ by \cite[Proposition 3.8]{MatCRELLE}, so that the full product is distinguished by \cite[Proposition 7.2]{OJNT} again.

Now we move on to the direct implications. We set $P = MU$ to be the standard parabolic subgroup of type $(m_1,m_2,\cdots,m_r)$. Each $\d_i$ is attached to a cuspidal segment $\D_i$ and we use the notation $\d_i=\d(\D_i)$. Given $s \in \CI_P$, we associate to $s$ the standard parabolic subgroup $P_s = M_sU_s \subset P$, $\theta_s$ and $\delta_s$ as in Section \ref{sec geom lem prelim}. The geometric lemma says that if $\pi$ is distinguished, then there is an $s \in \CI_P$ such that the normalized Jacquet module $r_{M_s,M} (\sigma) $ is $(M_s^{\theta_s}, \delta_s)$-distinguished. If we write for each $1 \leqslant i \leqslant r$, with obvious intuitive notations, 
\begin{align*}
	\D_i = [\D_{i r},\D_{i r-1} ,\cdots,\D_{i 1}]
\end{align*}
so that
\begin{align*}
	r_{M_s,M}(\sigma)  = \d(\D_{1 1}) \otimes \cdots \d(\D_{1 r}) \otimes \d(\D_{2 1}) \otimes \cdots \otimes \d(\D_{r r}),
\end{align*}
then putting $\d_{i j}:=\d(\D_{i j})$ we have
\begin{align}\label{eq::geom-lemma}
	\begin{cases*}
		\d_{i j} \cong \d_{j i}^{\vee},\quad i \neq j \\
		\d_{i i} \text{ is }(H_{m_{i i}^+,m_{i i}^-},\mu_i) \text{-distinguished,}
	\end{cases*}
\end{align}
where the character $\mu_i$ above could be described explicitly, but we only state their needed properties. We start with the following useful observations.

\begin{enumerate}
\item All characters $\mu_i$ are of the form $\xi_{\alpha_i}$ for some character $\alpha_i$ of $F^\times$.
\item\label{>1} If $m_{i i} > 1$ for some $i$, then both $m_{i i}^+$ and $m_{i i}^-$ are nonzero otherwise $\d_i$ would be a character by \eqref{eq::geom-lemma}, and $m_{i i}^+ = m_{i i}^-$ by Theorem \ref{thm p=q}. In particular, we don't care about the value of $\alpha_i$ thanks to Corollary \ref{cor sh vs twisted linear}, hence we can suppose that all $\mu_i$ are trivial, and moreover $\d_i$ is selfdual by the same corollary.
\item\label{=1} Denoting by $u_1<\cdots<u_t$ the indices such that $m_{u_k u_k} = 1 $, and setting \[\epsilon_k = m_{u_k u_k}^+ -  m_{u_k u_k}^- \in \{ 1, -1\}.\] Then $t$ is even and \[\sum_{k=1}^t \epsilon_k =0\] when $m$ is even whereas $t$ is odd and \[\sum_{k=1}^t \epsilon_k =1\] when $m$ is odd.  Moreover it follows from Lemma \ref{Proposition 3.6 MatCRELLE} that the characters $\mu_{u_i}$ are given by the formula 
\[\mu_{u_k}=\mu^{\frac{\e_k(-\sum_{i<k}\e_i+\sum_{j>k}\e_j)}{2}},\] which after some rewriting gives 
\begin{enumerate}
\item \[\mu_{u_k}=\mu^{\frac{1}{2}+\e_k(\sum_{j>k}\e_j)}\] when $m$ is even,
\item \[\mu_{u_k}=\mu^{\frac{1}{2}+\e_k(-\frac{1}{2}+\sum_{j>k}\e_j)}\] when $m$ is odd.
\end{enumerate}
 It then follows from the above formulae that the $\mu_i$'s satisfy the following inductive relations which are the correct version of \cite[Theorems 3.14, (2)]{MatCRELLE}:
 \begin{equation}\label{eq muui}
		\mu_{u_{k+1}} = \mu^{-1} \mu_{u_k} \quad \text{or}\quad \mu_{u_{k+1}} = \mu_{u_k}^{-1},
	\end{equation}
	depending on $\epsilon_{k+1} = \epsilon_k$ or not. Moreover,
	\begin{itemize}
	\item when $m$ is even: $\mu_{u_1} = \mu^{-1/2}$ and $\mu_{u_t} = \mu^{1/2}$. 
	\item when $m$ is odd: 
	\begin{itemize}
	\item $\mu_{u_1} = \mu^0=\triv$ if $m_{1 1}=m_{1 1}^+$,
	\item $\mu_{u_1} = \mu^{-1}$ if $m_{1 1}=m_{1 1}^-$.
	\item $\mu_{u_t} = \mu^0=\triv$ if $m_{t t}=m_{t t}^+$,
	\item $\mu_{u_t} = \mu^{-1}$ if $m_{t t}=m_{t t}^-$.
\end{itemize}
In particular setting $\mu'_{u_i}:=\mu_{u_{t+1-i}}^{-1}$, we see that the characters $\mu'_{u_i}$ satisfy the same relations and have the same initial values. This implies that the multiset $\{\mu_{u_1},\dots,\mu_{u_t}\}$ is stable under the map $\chi\mapsto \chi^{-1}$. However because it contains only positive characters, hence only the trivial character as a quadratic character, it can be partitioned by pairs $\{\chi,\chi^{-1}\}$ when $t$ is even, whereas it can be partitioned by such pairs together with one time $\triv$ when $t$ is odd. 
\end{itemize}
\end{enumerate}

We move on to the application of the geometric lemma. We can and will assume that $\pi$ is right-ordered, i.e. that the right ends $r(\D_i)$ of its cuspidal segments satisfy $e(r(\D_i))\geq e(r(\D_{i+1}))$. We now start to prove by induction on $n$ the following statement: If  $r_{M_s,M}(\sigma)$ is $(M_s^{\theta_s},\delta_s )$-distinguished, then $s $ is a monomial matrix, which will finish our proof thanks to the above observations. We have three cases.

Case (I)$\ $ If $m_{1 1} > 1$, then by the observation \eqref{>1} above $\d(\D_{1 1})$ is selfdual. By the right orderning of $\pi$, we have $\D_1 = \D_{1 1}$. We then can finish the proof by induction.

Case (II)$\ $ If $m_{1 1} = 1$ we use observation \eqref{=1} above. When $t$ is even, we have $r(\D_1) = \mu^{-1/2}$. This is absurd as the central character of $\pi$ will not be trivial. When $t$ is odd the only possible case is $m_{1 1}=m_{1 1}^+$ so $r(\D_1) = \triv$, which implies that $\D_1 = \D_{1 1}$ thanks to the right ordering, and we conclude by induction. 

Case (III)$\ $ If $m_{1 1} = 0$, then by the right ordering of $\pi$, we have $\D_1 = \D_{1 l}$ for some $l >1$. We need to show that $\D_l = \D_{l,1}$, and the rest follows from induction. If not so, then the only possibility is $m_{l,l} = 1$ and $\D_l =[\D_{l 1},\D_{l l}]$. We write $\D_1 = [\mu^a,\mu^b]$ with $b > 0$. So $\D_{l 1} = [\mu^{-b},\mu^{-a}]$ and $\D_{l l} = \mu^{-b-1}$. By the observation \eqref{eq muui} we know that there is a $u_k$ such that $\mu_{u_k} = \mu^c$ with $c > b+1$. This means that $\D_{u_k u_k} = \mu^c$. This is absurd as the segments are right-ordered.
\end{proof}

\subsection{Classification of standard modules with a Shalika model and $\chi$-distinction}

In this paragraph we prove Theorem \ref{thm class std shal} below, which is the classification of standard modules with a Shalika model. Denoting by $\re(\alpha)$ the unique real number such that if $\alpha:F^\times \to \BC^{\times}$ is a character, then 
\[|\alpha|=|\ |_F^{\re(\alpha)},\] we moreover extend the classification of distinguished standard modules to the case of $\chi$-distinguished ones, for $\chi$ of the form $\chi_{\alpha}$ with $\alpha$ such that $-\frac{d}{2}\leq \re(\alpha)\leq \frac{d}{2}$. 

We start by proving that certain type of induced representations admit a Shalika model with meromorphic families of Shalika functionals. Let 
$\pi_i$ be representations of $G_{m_i}$ for $i=1,\dots,l$ and $\underline{s}=(s_1,\dots,s_l) \in \BC^l$. Then if $\pi=\pi_1 \times \dots \times \pi_l$ is a representation of $G$, we set 
\[\pi[\underline{s}]=\nu^{ds_1}\pi_1 \times \dots \times \nu^{ds_l}\pi_l=\mu^{s_1}\pi_1 \times \dots \times \mu^{s_l}\pi_l.\] Iwasawa decomposition allows one to realize all these representations on the same vector space, and we thus refer to \cite[2.5]{MO} for the definition of a meromorphic family in the variable $\underline{s}$ of linear forms on $\pi[\underline{s}]$. 

\begin{prop}\label{prop mero shal 1}
Let $\d$ be a discrete representations of $G_m$, then there is a nonzero meromorphic family $(\lambda_s)_{s\in \BC}$ with \[\lambda_s\in \Hom_{S_{2m}}(\d[s]\times \d^\vee[-s],\Theta).\] In particular $\d[s]\times \d^\vee[-s]$ admits a Shalika model for all $s\in \BC$.
\end{prop}
\begin{proof}
The proof is exactly the same as the one given in the proof of \cite[Lemma 3.11]{MatCRELLE} together with its correction 
in \cite[Proposition 5.3]{MatBF2}, and relies on the Bernstein principle of meromorphic continuation of equivariant linear forms in a generic 
multiplicity one situation, and the multiplicity at most one of Shalika models for irreducible representations proved in Corollary \ref{cor shal mult 1}.
\end{proof}
\begin{rem}
In general, the fact that any representation of the form $\pi\times \pi^\vee$ has a Shalika model follows immediately from the computation of 
twisted Jacquet modules in \cite[Theorem 2.1]{PR}.
\end{rem}

\begin{rem}\label{rem unitary}
The representation $\delta$ could be replaced by any irreducible unitary representation as follows from its proof.
\end{rem}

Here is another result of the same flavor.

\begin{prop}\label{prop mero shal 2}
Let $\pi_i$, resp. $\pi'_j$, be representations of $G_{m_i}$, resp. $G_{m'_j}$, for $i=1,\dots,l$, resp. $j=1,\dots,l'$, and set 
\[\pi=\pi_1 \times \dots \times \pi_l\] and \[\pi'=\pi'_1 \times \dots \times \pi'_{l'}\] to be the corresponding induced representations of $G_m$ and $G_m'$ respectively. Let $A\subseteq \BC^{l}$, and $A'\subseteq \BC^{l'}$ be two complex vector subspaces. Suppose that $m$ and $m'$ are even, and that there are two nonzero meromorphic families $(\lambda_{\underline{s}})_{\underline{s}\in A}$ and $(\lambda'_{\underline{s'}})_{\underline{s'}\in A'}$ with \[\lambda_{\underline{s}}\in \Hom_{S_{m}}(\pi[\underline{s}],\Theta)\] and \[\lambda_{\underline{s'}}\in \Hom_{S_{m'}}(\pi'[\underline{s'}],\Theta).\] Then there is a nonzero meromorphic family 
$(\Lambda_{\underline{s},\underline{s'}})_{(\underline{s},\underline{s'})\in A\times A'}$ in \[\Hom_{S_{m+m'}}(\pi[\underline{s}]\times \pi'[\underline{s'}],\Theta).\]
\end{prop}
\begin{proof}
The proof follows from the explicit construction of the Shalika functional on $\pi[\underline{s}]\times \pi'[\underline{s'}]$ given in \cite{MatBLMS} which is valid for inner forms, with the following observations. First there is a quantifier ordering problem in \cite[Lemma 3.2]{MatBLMS}: the compact set $C$ depends on the function $f$ there. However in our situation, this compact set will be independent of $(\underline{s},\underline{s'})\in A\times A'$ for a given holomorphic section $f_{ \underline{s},\underline{s'} }$ as follows from the proof of Propositions 4.1 and 4.2 there. In particular with the notations of \cite[Lemma 3.2]{MatBLMS} the map $\phi_{\underline{s},\underline{s'} }(f_{ \underline{s},\underline{s'} })$ is meromorphic in the variable $(\underline{s},\underline{s'})$ as follows from the 
first equality in the proof of the lemma. More precisely it follows from the assumption on $\lambda_{\underline{s}}$ and $\lambda_{\underline{s'}}$, there exists a nonzero rational map $R(\underline{s},\underline{s}')$ in the variables $q_F^{-s_i}$ and $q_F^{-s'_j}$ such that $R(\underline{s},\underline{s}')\lambda_{\underline{s},\underline{s}'}$ is holomorphic and nonzero, where we set 
\[\lambda_{\underline{s},\underline{s}'}:=\lambda_{\underline{s}}\otimes \lambda_{\underline{s'}}\] as in the lemma. In particular according to the aforementioned equality, we deduce \[R(\underline{s},\underline{s}')\phi_{ \underline{s},\underline{s'} }(f_{ \underline{s},\underline{s'} })\] is holomorphic for any choice of $f$. Later in the proof of the lemma, the map $R(\underline{s},\underline{s}')\Phi_{ \underline{s},\underline{s'} }(f_{ \underline{s},\underline{s'} })$ will thus be holomorphic as well, by compactness of the quotient 
$P'\backslash G'$ there. Finally it will be nonzero for some choice of $f$ by the end of the proof of the same lemma.
\end{proof}

We are now in position to prove the first part of the theorem of interest to us.

\begin{prop}\label{prop st shal prelim}
Let $\tau$ be an irreducible representation with a Shalika model, and $\d_1 \times \dots \times \d_l$ be discrete series representations. Then the induced representation 
\[\pi_{\underline{s}}:=\d_1[s_1] \times \dots \times \d_l[s_l] \times \tau \times \d_l^\vee[-s_l] \times \dots \times \d_1^\vee[-s_1]\] admits a Shalika model for all values of $\underline{s}\in \BC^l.$
\end{prop}
\begin{proof}
By Propositions \ref{prop mero shal 1} and \ref{prop mero shal 2}, the representation 
\[\pi'_{\underline{s}}:=\tau \times \d_1[s_1] \times \d_1^\vee[-s_1] \times \dots \times \d_l[s_l] \times \d_l^\vee[-s_l]\] admits a nonzero meromorphic family of 
Shalika functionals on its space. It then follows from the meromorphy and generic invertibility of the standard intertwining operator from $\pi_{\underline{s}}$ to $\pi'_{\underline{s}}$ that so does the representation $\pi_{\underline{s}}$. The result now follows from a usual leading term argument (see Section \ref{sec sg open periods} for more details on such arguments).  
\end{proof}
\begin{rem}
It also follows from the proof above that the order of the representations in the product of the statement of the proposition could be permuted as one likes, but the same conclusion would still hold.
\end{rem}
\begin{rem}
One could also only suppose the representations $\d_i$ to be unitary in the above statement according to Remark \ref{rem unitary}.
\end{rem}

As a corollary, we obtain the following theorem:

\begin{thm}\label{thm class std shal}
A standard module of $G=G_m$ has a Shalika model if and only if it is $H_{m/2,m/2}$-distinguished (see Theorem \ref{thm linear std}).
\end{thm}
\begin{proof}
The right to left implication follows from Theorem \ref{thm linear std} and Proposition \ref{prop st shal prelim}. The other direction follows from Corollary \ref{cor shal mult 1}.
\end{proof}

Finally we end with a word on $\chi$-distinction for standard modules. As observed in \cite[Proposition 3.4]{su24}, the proof of Theorem \ref{thm linear std} holds with no modification in the case of twisted linear periods when one adds a twisting character $\chi$ (the reason being that the so called modulus assumption stated in \cite{MOY} holds in this situation), and we refer to it for the statement of the classification of $\chi$-distinguished standard modules in this case. Here we give similar statements in the case of linear models for 
$H_{m/2,m/2}$ and $H_{(m+1)/2,(m-1)/2}$.

\begin{thm}\label{thm chi linear std}
Let $\pi$, $\d_i$ and $\pi_i$ be as in the statement of Theorem \ref{thm linear std}, and $\alpha$ be a character of $F^\times$. 
\begin{enumerate}
\item\label{u} If $m$ is even and $\pi$ is $H_{m/2,m/2}$-distinguished, then it is $(H_{m/2,m/2},\xi_{\alpha})$-distinguished, and if moreover 
$-\frac{d}{2}\leq \re(\alpha) \leq \frac{d}{2}$ the converse implication holds. 
\item\label{v} If $m$ is odd, $-\frac{d}{2}\leq \re(\alpha)< \frac{d}{2}$ and $\pi$ is $(H_{(m+1)/2,(m-1)/2},\xi_{\alpha})$-distinguished, then there exists an index $i_0$ such that $\d_{i_0}=\alpha\circ \nu$, and $\pi_{i_0}$ is $H=H_{(m-1)/2,(m-1)/2}$-distinguished. 
\end{enumerate}
\end{thm}
\begin{proof}
The first implication of \eqref{u} follows from Theorem \ref{thm class std shal}. The proof of the converse implications of \eqref{u} and \eqref{u} are completely similar to the proofs for trivial $\alpha$ given in Theorem \ref{thm linear std}, and we briefly explain the modifications to be done. 
In equation \eqref{eq::geom-lemma} the character $\mu_i$ must be replaced by $\mu_i\chi_{\alpha}$ with the same definition of $\mu_i$. 
Let us set \[\nu_{\alpha}:=\alpha\circ \nu_{D^\times}.\] 

The proofs of Cases (I) and (III) are the same as when $\alpha=\triv$.

In the proof of Case (II) when $m$ is even, the character $r(\D_1)$ must be replaced by $\nu_{\alpha}^{\pm 1} \mu^{-1/2}=\nu_{\alpha}^{\pm 1}\nu^{-d/2}$ where the sign $\pm 1$ is positive if $m_{1,1}=m_{1,1}^+$ and negative if $m_{1,1}=m_{1,1}^-$. According to the case the argument when 
$\re(\alpha)<d/2$ or $-\re(\alpha)<d/2$ is the same. When $m_{1,1}=m_{1,1}^+$ and $\re(\alpha)=d/2$ or $m_{1,1}=m_{1,1}^-$ and $\re(\alpha)=-d/2$, then the argument is different. In this situation, because the central character is trivial and the standard module is right ordered, all $m_i$'s are equal to $1$, and all $\mu_{u_i}$ must be unitary. The $m_k$'s either come in pairs  $(m_{i,j},m_{j,i})$ for $i<j$, and for the remaining $m_{u_k,u_k}$ from $k=1,\dots,t$, the first case of Equation \eqref{eq muui} can never happen, hence the sign of $\e_k$ changes at each step. We deduce that $\pi$ is of the form 
\[\pi=(\nu_{\alpha|\ |^{-d/2}}\times \nu_{\alpha|\ |^{-d/2}}^{-1})^{t/2}\times \prod_{j=1}^{(m-t)/2}(\chi_j\times \chi_j^{-1})\] for some unitary characters $\chi_j$ of $D^\times$ when $m_{1,1}=m_{1,1}^+$, whereas it is of the form \[\pi=(\nu_{\alpha|\ |^{d/2}}\times \nu_{\alpha|\ |^{d/2}}^{-1})^{t/2}\times \prod_{j=1}^{(m-t)/2}(\chi_j\times \chi_j^{-1})\] when $m_{1,1}=m_{1,1}^-$. 

We now discuss the  proof of Case (II) when $m$ is odd . The character $r(\D_1)$ must be replaced by $\nu_{\alpha}$ when $m_{1,1}=m_{1,1}^+$ and by $\nu_{\alpha^{-1}|\ |^{-d}}$ when $m_{1,1}=m_{1,1}^-$, but this second case cannot happen for right-ordering and central character reasons. When $m_{1,1}=m_{1,1}^+$ and $\re(\alpha)<d/2$ the argument is the same as in the case of trivial $\alpha$. 

\end{proof}

We take the opportunity to extend \cite[Theorems 3.13]{MatCRELLE} in the split case to the case of inner forms. When $m$ is even we obtain the same statement as in the case \eqref{u} of standard modules in Theorem \ref{thm chi linear std}, as a special case of it. When $m$ is odd we obtain the following statement.
 
\begin{thm}\label{thm chi linear gen}
Let $\pi$, $\d_i$ and $\pi_i$ be as in in the statement of Theorem \ref{thm linear std}, and $\alpha$ be a character of $F^\times$ with $-\frac{d}{2}\leq \re(\alpha)< \frac{d}{2}$. Suppose moreover that $\pi$ is generic and that $m$ is odd. Then $\pi$ is $(H_{(m+1)/2,(m-1)/2},\chi_{\alpha})$-distinguished if and only if there exists a generic $H_{(m-1)/2,(m-1)/2}$-distinguished generic representation $\pi'$ of $G_{m-1}$ such thate $\pi=\pi'\times \nu_{\alpha}$.
\end{thm}
\begin{proof}
By Theorem \ref{thm chi linear std}, if $\pi'$ is $H_{(m-1)/2,(m-1)/2}$-distinguished, it is also $\mu^{1/2}\chi_{\alpha}$-distinguished thanks to Theorem \ref{thm chi linear std}, hence $\pi=\pi'\times \nu_{\alpha}$ is $\chi_{\alpha}$-distinguished by 
\cite[Proposition 3.8]{MatCRELLE}. The converse implication follows from Theorem \ref{thm chi linear std}.
\end{proof}
\begin{rem}
When $G$ is split, one can allow $\re(\alpha)=\frac{1}{2}$ in the statement of Theorem \ref{thm chi linear gen} above because $\pi$ is $\chi$-distinguished if and only if $\pi^\vee$ is $\chi^{-1}$-distinguished by the realization of the contragredient of an irreducible representation using transpose inverse. 
\end{rem}

Together with Proposition \ref{prop symp disc} we obtain:

\begin{cor}\label{cor symp gen}
Suppose that $\pi$ is an irreducible representation of $G$ with either a twisted linear period, or a linear period with respect to 
$H_{l,l,}$ for $|l-l'|\leq 1$. Then it is selfdual. If moreover $\pi$ is generic, and $\pi$ has a twisted linear period or a linear period with respect to $H_{m,m}$, then it is symplectic. 
\end{cor}
\begin{proof}
Only the first statement requires a proof, but it is an immediate consequence of Theorem \ref{thm linear std} and Proposition \ref{prop symp disc}, together with the Langlands quotient theorem.
\end{proof}

\section{Sign and parabolic induction}\label{sec sign and induction}

By convention $G_0$ is the trivial group, and the sign of any representation of $G_0$ is equal to $1$. We study the sign of parabolically induced representations, using closed and open intertwining periods. This natural strategy is not new as it was already used in \cite[Lemma 3.5]{LM} in the case of linear periods for split $G$, as was pointed to us by Lapid. 

\subsection{Sign and irreducible parabolic induction}

\begin{prop}\label{prop closed sign}
If $\pi_1$, $\pi_2$ are distinguished representations of $G_{m_1}$ and $G_{m_2}$ for positive integers $m_i$, and $\pi_1\times \pi_2$ is irreducible, then 
\[\sg(\pi_1\times \pi_2)=\sg(\pi_1)\sg(\pi_2).\]
\end{prop}
\begin{proof}
We set $m=m_1+m_2$, and let $L_i$ a nonzero $H_i$-invariant linear form on $\pi_i$, where $H_i$ is the relevant subgroup of $G_i$, 
i.e. either $H_{m_i/2,m_i/2}$ if we consider linear periods, or the centralizer of $\d$ in $G_{m_i}$ with square in the non-square elements of $G_{m_i}$ in the case of twisted linear periods. The $H$-invariant linear form on $\pi_1\times \pi_2$ is given by the closed intertwining period 
\[f\mapsto \int_{P_{m_1,m_2}\cap H\backslash H} L_1\otimes L_2(f(h)) dh.\] In both the non twisted and twisted case, we set $u_i$ as in Section \ref{sec normalizer} such that the normalizer $N_i$ of $H_i$ is equal to $H_i\sqcup u_i H_i$. In particular $u:=\diag(u_1,u_2)$ is such that $N=H\sqcup u H$, and we observe that $\d_{P_{m_1,m_2}}(u)=1$. The result now follows from the change of variable 
\[\int_{P_{m_1,m_2}\cap H\backslash H} L_1\otimes L_2(f(hu^{-1})) dh=\int_{P_{m_1,m_2}\cap H\backslash H} L_1\otimes L_2(f(u^{-1}h)) dh\]
\[=\int_{P_{m_1,m_2}\cap H\backslash H} L_1\otimes L_2(\pi_1(u_1^{-1})\otimes \pi_2(u_2^{-1})f(h)) dh\]
\[=\chi_{\pi_1}(\overline{u_1})\chi_{\pi_2}(\overline{u_2})\int_{P_{m_1,m_2}\cap H\backslash H} L_1\otimes L_2(f(h)) dh,\] and the identity 
\[\e(\pi_1\times\pi_2)=\e(\pi_1)\e(\pi_2).\]  
\end{proof}

\subsection{Sign and open intertwining periods}\label{sec sg open periods}

Before stating the next results, we need to extend somehow the definition of the character $\chi_{\pi}$. So far we have only considered irreducible representations for which multiplicity one of $H$-invariant linear form is guaranteed. In particular it made sense to talk about $\chi_{\pi}$ rather than $\chi_{L}$ for $L$ the (possibly twisted) linear period on it. Here, though we are only interested in irreducible representations, we need to consider some which are not, though of finite length. On such a finite length representation $\pi$ of $G$ one might loose multiplicity one, so we prefer to talk about the character of $\frac{N}{H}$ attached to an $H$-invariant linear form on $\pi$ when it exists. Hence if $L\in \Hom_H(\pi,\BC)-\{0\}$ is such that there exists $z$ in $\BC^\times$ satisfying $L\circ \pi(u^{-1})=z.L$, we define $\chi_L$ to be the character of $\frac{N}{H}$ such that 
\[\chi_L(\overline{u}^{-1})=z.\] In particular $\chi_L=\chi_{\pi}$ if $\pi$ is irreducible. 

Now we recall the Blanc-Delorme theory of 
open intertwining periods (see \cite{BD}). For $\pi$ a finite length representation of $G$ and $s\in \BC$, we recall our notation \[\pi[s]:=\mu^s\pi.\] Let $\pi_0,\pi_1,\dots,\pi_r$ be finite length representations of $G_{m_0},G_{m_1},\dots,G_{m_r}$ respectively and suppose that $\pi_0$ possesses an $H$-invariant linear form $\ell_0$. Then the representation 
\[\pi_{\underline{s}}:=\pi_r[s_r]\times \dots \times \pi_1[s_1]\times \pi_0\times  \pi_1^\vee [-s_1]\times \dots \times  \pi_r^\vee [-s_r],\] is distinguished for any $s_i\in \BC$ and there is an explicit enough description of it which we now recall. 

We first make specific choices of the element $\delta$, hence of the element $u$ and the group $H$, well adapted to describe  situation in a simple manner, and which we will also use in the next section:

\begin{enumerate}[1)]
\item In the case of linear periods we take 
\[\delta:=\begin{pmatrix} &  & & & & &        & I_{m_r} \\
                         &  & & & & & \iddots &  \\ 
												&  & & & & I_{m_1} &  &  \\ 
									        &  & & I_{m_0/2} &  &     &   &  \\ 
													  &  & &   & -I_{m_0/2} &     &   &  \\ 
												    & & I_{m_1} & & &   &  & \\ 
												    & \iddots& & & &   &  &  \\ 
												  I_{m_r} & & & & &   &  &  \\ 
																							\end{pmatrix} \] and
																							
		\[u=\diag(I_{m_r},\dots,I_{m_1},I_{m_0/2},-I_{m_0/2},-I_{m_1},\dots,-I_{m_r}). \]
		\item 
\begin{enumerate} 
\item In the case of odd twisted linear periods we take 
\[\delta:=\begin{pmatrix}   & & & & &        & \k I_{m_r} \\
                          & & & & & \iddots &  \\ 
												  & & & & \k I_{m_1} &  &  \\ 
									          & & &\delta_{D^\times}  I_{m_0} &     &   &  \\
												    & & I_{m_1} & & &   &   \\ 
												    & \iddots& & & &   &    \\ 
												  I_{m_r} & & & & &   &    \\ 
																							\end{pmatrix} \] and
																							
		\[u=\diag(I_{m_r},\dots, I_{m_1},\iota I_{m_0},-I_{m_1},\dots,-I_{m_r}). \]
\item In the case of even twisted linear periods we take 
\[\delta:=\begin{pmatrix} &  & & & & &        & \k I_{m_r} \\
                         &  & & & & & \iddots &  \\ 
												&  & & & & \k I_{m_1} &  &  \\ 
									        &  & & & \k I_{m_0/2} &     &   &  \\ 
													  &  & &  I_{m_0/2} & &     &   &  \\ 
												    & & I_{m_1} & & &   &  & \\ 
												    & \iddots& & & &   &  &  \\ 
												  I_{m_r} & & & & &   &  &  \\ 
																							\end{pmatrix} \] and
																							
		\[u=\diag(I_{m_r},\dots,I_{m_1},I_{m_0/2},-I_{m_0/2},-I_{m_1},\dots,-I_{m_r}). \]
\end{enumerate}
\end{enumerate}

Now put \[G:=G_{m_0+2\sum_{i=1}^r m_i},\] 
\[P:=P_{(m_r,\dots, m_1,m_0,m_1,\dots,m_r)}\]  
\[M:=M_{(m_r,\dots, m_1,m_0,m_1,\dots,m_r)},\] and observe that with the choices above, the parabolic subgroup $P$ 
is $\theta$-split, in particular $P^\theta=M^\theta$. 

We denote by $\ell$ the linear form on \[\pi_r \otimes \dots \otimes \pi_1 \otimes \pi_0\otimes  \pi_1^\vee  \otimes \dots \otimes  \pi_r^\vee  \] defined by 
\[\ell(v_r\otimes \dots \otimes v_1 \otimes  v_0 \otimes v_1^\vee \otimes \dots\otimes v_r^\vee )
=\ell_0(v_0)\prod_{i=1}^r v_i^\vee(v_i).\] 
Then if $f_{\underline{s}}$ is a holomorphic section of $\pi_{\underline{s}}$, the integral 
\[I_{{\underline{s}}}(f_{\underline{s}}):=\int_{P^{\theta}\backslash H} \ell(f_{\underline{s}}(h))dh\] converges for $s_{i+1}-s_i$ of real part 
large enough, and extends meromorphically. For a generic choice of nonzero vector $\underline{a}=(a_1,\dots,a_r)\in \BC^r$ and an integer $l$ such that for any $f:=f_0\in \pi:=\pi_0$, the limit 
\[L_{\underline{a}}(f):=\lim_{s\rightarrow 0} s^l I_{s\underline{a}}(f_{s\underline{a}})\] is well defined and nonzero for some $f$. In short we will say that 
$L_{\underline{a}}$ is a leading term of $I_{\underline{s}}$ at $s=0$.

\begin{prop}\label{prop open sign}
Let $\pi_0$ be a finite length representation of $G_{m_0}$, where $m_0$ is a non negative integer, which admits a nonzero $H$-invariant linear form $\ell_0$ which admits a sign $\chi_{\ell_0}$. Let $\pi_i$ be a finite length representation of $G_{m_i}$, for $m_i\geq 1$, with a central character for $i=1,\dots,r$. Then the open intertwining period $L$ on 
\[\pi:=\pi_r \times \dots \times \pi_1 \times \pi_0\times  \pi_1^\vee  \times \dots \times  \pi_r^\vee \] associated to $\ell_0$ as above has a sign character $\chi_L$, moreover  
\[\frac{\chi_L}{\prod_{i=1}^r\omega_{\pi_i}(-1)}=\chi_{\ell_0}.\] In particular when $\pi$ is irreducible:
\[\sg(\pi)=\sg(\pi_0).\] 
\end{prop}
\begin{proof}

By induction, up to the replacements \[\pi_0:=\pi_{r-1} \times \dots \times \pi_1 \times \pi_0\times \pi_1^\vee \times \dots \times  \pi_{r-1}^\vee \] and 
\[\ell_0(v_0):=\ell_0(v_0)\prod_{i=1}^{r-1} v_i^\vee(v_i),\] we are reduced to the $r=1$ case, which we assume for the rest of this proof.

Let $f_s$ be a holomorphic section of \[\pi_s:=\pi_1[s]\times \pi_0\times  \pi_1^\vee [-s].\]
The $H$-invariant linear form in this case is given by the leading term at $s=0$ of the following open intertwining period: 
\[I_s(f_s)=\int_{P^{\theta}\backslash H} \ell(f_s(h))dh.\] Now we observe that 
\[\ell(\pi_1(I_{m_1})v\otimes\pi_1^\vee(\theta(-I_{m_1}))v^\vee)
=\omega_{\pi_1}(-1)\ell(v \otimes v^\vee).\] 
Hence by  a change of variable in the integral defining $I_s$ for $\re(s)$ large enough we obtain: 
\[\chi_{I_s}=\omega_{\pi_1}(-1)\chi_{\ell_0}.\] 
The first part of statement then follows by a meromorphic continuity argument. For the second part, if 
$\pi$ is irreducible then so is $\pi_1$. Hence \[\e(\pi_1\times \pi_1^{\vee})=\omega_{\pi_1}(-1).\]
\end{proof}

\subsection{Sign of discrete series representations}

In this section it turns out that for linear periods, the Steinberg representation 
$\St_m(\triv_{D^\times})$, where $m$ has to be even, exhibits an exceptional behaviour. First its root number is not given by the same formula as the other generalized Steinberg representations with a linear period, and moreover its 
$H$-invariant linear form does not come from the open intertwining period on the induced representation lying above it. Hence 
we first prove our general result for discrete series representations in the special case of linear periods for the Steinberg representation.

\begin{prop}\label{thm st}
Let $m$ be an even integer. Then the Steinberg representation $\pi:=\St_m(\triv_{D^\times})$ has no twisted linear period. On the other hand it has a linear period, and moreover its sign for linear period is 
\[\sg(\pi)=1.\]
\end{prop}
\begin{proof}
The first assertion follows from \cite[Theorem 0.1]{Ch}. We prove the second in two steps. We start with $m=2$. Then the $H$-invariant linear form $\ell_0$ on $\pi$ comes from the descent of the difference of the two closed intertwining periods on the induced representation $\pi_2(\triv_{D^\times})$. Namely 
\[\ell_0(f)=f(I_2)-f(u)\] where \[u:=\begin{pmatrix} & 1 \\ 1 & \end{pmatrix}.\] Indeed this linear form clearly vanishes on the spherical vector. Hence $\chi_{\St_2(\triv_{D^\times})}(\overline{u})=-1$ as $u^2=I_2$. On the other hand 
$\e(\St_2(\triv_{D^\times}))=-1$ and the result for $m=2$ follows. 

So we suppose that $m\geq 4$. Now for $k$ between $1$ and $m-1$, set 
\[\tau_k:=\triv_{D^\times}\times'\dots\times' \triv_{D^\times}\times'\underbrace{\triv_{G_2}}_{position \ k} \times' \triv_{D^\times}\times' \dots  \times' \triv_{D^\times}\] and 
\[\sigma :=\triv_{D^\times}\times'\dots\times' \triv_{D^\times}\times'\underbrace{\St_2(\triv_{D^\times})}_{middle} \times' \triv_{D^\times}\times' \dots  \times' \triv_{D^\times}\] where $\times'$ stands for non-normalized parabolic induction. We denote by 
$\overline{\tau_k}$ the image of $\tau_k$ in the quotient $\sigma$, and we recall that \[\St_m(\triv_{D^\times})\simeq 
\sigma /\sum_{k\neq m/2} \overline{\tau_k}.\] Arguing as in \cite[Lemma 10.12]{MatJFA}, one checks that 
no $\tau_k$ for $k\neq m/2$ is distinguished. By multiplicity one on $\St_m(\triv_{D^\times})$, this proves on the one hand that $\sigma$ has a unique $H$-invariant linear form, and on the other hand that the $H$-invariant linear form on 
$\St_m(\triv_{D^\times})$ is its descent. Now the $H$-invariant linear form $L$ on $\sigma_k$ is a leading term of the open period attached to $\ell_0$ as in the proof of Proposition \ref{prop open sign}, and it follows from this proposition that 
\[\chi_L=\chi_{\ell_0}.\] 
Hence \[\chi_{\St_m(\triv_{D^\times})}=\chi_{\St_2(\triv_{D^\times})}.\] 
On the other hand, by the formula for the root number of the Steinberg representation, we have 
\[\e(\St_m(\triv_{D^\times}))=\e(\St_2(\triv_{D^\times}))\] and the result follows.
\end{proof}

The sign of cuspidal representations is computed by a global method in Theorem \ref{thm cusp}. Its proof uses the results above and a globalization argument. Propositions \ref{prop closed sign}, \ref{prop open sign} and Theorem \ref{thm st} allow to reduce the proof for discrete series representations to the cuspidal case, and we explain this now. 

\begin{thm}\label{thm disc}
Let $\pi$ be a distinguished discrete series representation of $\GL_m(D)$. Then \[\sg(\pi)=(-1)^m.\] 
\end{thm}
\begin{proof}
Thanks to Proposition \ref{thm st}, we can assume that $\pi=\St_t(\rho)$ for $\rho\neq \triv_{D^\times}$. In particular if we write $\JL(\rho)=\St_r(\rho')$, we have 
\[\e(\pi)=\e(\rho')^{tr}=\e(\rho)^{t},\] On the other hand the $H$-invariant linear form on $\pi_t(\rho)$ is the descent of the unique open period on $\pi_t(\rho)$ according to Proposition \ref{prop st implies cusp} (in fact this open intertwining period has to be holomorphic at the parameter giving $\pi_t(\rho)$ by a reasoning similar to those done in \cite{MatJFA}). By Proposition \ref{prop open sign}, when $t$ is odd we have 
\[\pi=\St_t(\rho)\Rightarrow\chi_{\St_t(\rho)}=\e(\rho)^{(t-1)}\chi_{\rho},\] and the result follows from Theorem \ref{thm cusp}. When $t$ is even, fix any non-trivial character $\psi$ of $F$. By Proposition \ref{prop open sign} we obtain \[\chi_{\St_t(\rho)}(\overline{u})=\omega_{\rho}(-1)^{t/2}=[\e(\rho,\psi)\e(\rho,\psi)]^{t/2}=\e(\rho,\psi)^t=\e(St_t(\rho)),\] and the sign is $1$ as expected in this case.
\end{proof}

\section{The sign of cuspidal representations}\label{sec cusp}
 
\subsection{The sign of distinguished generic unramified representations for split $G$}\label{sec gen UR sign}

Here we assume that $G=\GL_n(F)$ is $F$-split (hence $n$ is even) and we set $K:=\GL_n(O_F)$ where $O_F$ is the ring of integers of $F$. In the case of linear models, the result below follows from \cite[Theorem 5.1]{MatBLMS} and \cite[Proposition 3.2]{FJ} and its proof, hence we only prove it for twisted linear periods. 
 
\begin{prop}\label{prop UR sign}
Suppose that $\pi$ is generic unramified and distinguished representation of $G$. Let $v_0$ be a $K$-spherical vector in $\pi$, and $\Lambda\in \Hom_H(\pi,\BC)-\{0\}$. Then 
$\Lambda(v_0)\neq 0$, hence $\sg(\pi)=1$.
\end{prop}
\begin{proof}
By Theorem \ref{thm linear std}, and because $\pi$ is irreducible, the representation can be written under the form $\pi=\chi_1\times \cdots\times \chi_{2k}$ with $n=2k$ for unramified characters $\chi_i$ of $F^\times$ such that $\chi_{2(i+1)}=\chi_{2i+1}^{-1}$. In this situation each representation $\pi_i:\chi_{2i+1}\times \chi_{2(i+1)}$ of $\GL_2(F)$ is $E^\times$-distinguished with distinguishing linear form $L_i$. Then by the proof of Proposition \ref{prop closed sign}, setting \[L:=\otimes_{i=1}^k L_i:\pi_1\otimes \dots \otimes \pi_k\rightarrow \BC,\] 
the $H$-invariant linear form on $\pi=\prod_{i=1}^k \pi_i$ is given by 
\[\Lambda: f\mapsto \int_{B_H\backslash H} L(f(h))dh\sim \int_{K_H} L(f(h))dh .\] Here the symbol \say{$\sim$} means equality up to a positive constant
and $B_H$ is the intersection of the upper triangular Borel of $G$ with $H$, i.e. the upper triangular Borel of $H$, and $K_H$ is $\GL_k(O_E)$. 

The (unique up to scalar in $\BC^\times$) spherical vector $f_0\in \pi$ is given by 
\[f_0(pk)=\delta_{P_{(2,\cdots,2)}}(p)^{1/2}v_1\otimes \dots \otimes v_k\] where each $v_i$ is the spherical vector of $\pi_i$. Hence up to appropriate normalizations: 
\[\Lambda(f_0)=\prod_i L_i(v_i).\]
This reduces our problem to the $n=2$ case, and we write $\pi=\chi\times \chi^{-1}$. Now we consider separately the case where $E/F$ is unramified and the case where $E/F$ is ramified. \\

If $E/F$ is unramified, then one can take $\delta\in U_E$, so that with our choice of embedding $U_E\leq \GL_2(O_F)$ 
and the well defined $E^\times$-invariant linear form \[\Lambda:f\mapsto \int_{F^\times \backslash E^\times} f(h)dh\] on 
$\pi$ is nonzero on the spherical vector $f_0$ as 
\[\Lambda(f_0)=\int_{U_F \backslash U_E} f_0(h)dh=vol(U_F \backslash U_E)>0.\]

If $E/F$ is ramified, then $\delta$ can be chosen as the uniformizer $\begin{pmatrix} 0 & \kappa  \\1 & 0 \end{pmatrix}$ of $E$, and 
\[E^\times=F^\times U_E \sqcup  F^\times \begin{pmatrix} 0 & \kappa  \\1 & 0 \end{pmatrix}U_E .\] Because the Iwasawa decomposition of $\begin{pmatrix} 0 & \kappa  \\ 1 & 0 \end{pmatrix}$ is 
$\begin{pmatrix} \kappa & 0 \\ 0 & 1 \end{pmatrix}\begin{pmatrix} 0 & 1  \\ 1 & 0 \end{pmatrix}$, we obtain 

\[\Lambda(f_0)=\int_{U_F \backslash U_E} f_0(h)dh+\int_{U_F \backslash U_E} f_0\left(\begin{pmatrix} \kappa & 0 \\ 0 & 1 \end{pmatrix}h\right)dh=(1+q^{1/2}\chi(\kappa))vol(U_F \backslash U_E).\] This proves the expected result, except when the unramified character $\chi$ satisfies $\chi(\kappa)=-q^{-1/2}$, which implies $\chi=\nu_F \chi^{-1}$. This would imply that $\pi$ is reducible, hence this case does not occur.

Finally, we conclude that $\sg(\pi)=1$ because on one hand we just proved that $\chi_{\pi}$ is trivial, but on the other hand 
$\epsilon(\pi)=1$.
\end{proof}

\subsection{The sign of linear periods: generic unitary representations for split $G$}\label{sec gen sign}

Here $G$ is $F$-split (i.e. $m=n$) and $H=L_{n/2,n/2}$. Moreover $F$ is allowed to be Archimedean as well. Because of a later globalization process, we are at the moment only interested in local components of cuspidal automorphic representations with cuspidal Jacquet-Langlands transfer which admit a global twisted linear period. However such representations are known to admit a global Shalika model as we shall recall later. Hence all their local components admit a local Shalika model.

So in this paragraph, we suppose that our local representation $\pi$ of $G$ admits a Shalika model. By \cite{JR} this implies that $\pi$ has a linear model, and we will recall the argument in the proof below when $\pi$ is generic unitary. In fact, when $F$ is $p$-adic, it is known by \cite{MatBLMS} or \cite{Gan} that a generic irreducible representation $\pi$ of $G$ has a linear model if and only if it has a Shalika model. This could certainly be proved in the Archimedean setting as well due to the advanced stage of the literature in this setting, but we don't need this result here.

In any case, under this assumption, the computation of the sign of local linear periods immediately follows from the results in \cite{FJ}.

\begin{prop}\label{prop generic linear sign}
Let $\pi$ be an irreducible generic unitary (of Casselman-Wallach type when $F$ is Archimedean) representation of $G$ with a Shalika model. Then it has a linear period and 
$\sg(\pi)=1$. 
\end{prop}
\begin{proof}
Our result is just a re-interpretation of a result of Friedberg and Jacquet. They prove in \cite[Proposition 3.1]{FJ} that if $S$ is in the Shalika model of $\pi$, the integral 
\[\Psi(s,S)=\int_{G_n} S(\diag(g,I_n))|\det(g)|_F^{s-1/2}dg\] initially defined for $s$ of real part large enough in fact extends to a meromorphic function on $\BC$, which is a holomorphic multiple of $L(s,\pi)$, and is equal to it for some choice of $S$. They moreover prove in \cite[Proposition 3.6]{FJ} the functional equation 
\[\Psi(1-s,\widetilde{S})=\gamma(s,\pi)\Psi(s,S).\] where 
\[\widetilde{S}(g)=S(ug^{-t}).\] Under our assumptions $L(1/2,\pi)$ is finite (and of course nonzero) so 
$\Psi(1/2,S)$ makes sense for any $S$ and because $\pi$ is selfdual:

\[\Psi(1/2,\widetilde{S})=\epsilon(1/2,\pi)\Psi(1/2,S).\] 

Now observe that \[L:S\mapsto \Psi(1/2,S)\] is $H$-invariant, and that it is moreover nonzero as $L(1/2,\pi)\neq 0$. 
Finally up to easy changes of variable, the left hand side of the above equation exactly expresses the action of the normalizer $N$ on the linear form $L$, and our claim follows.
\end{proof}

\subsection{The global sign of twisted linear periods}

Here we fix a number field $k$ with its adele ring $\BA$, and denote by $\CP(k)$ the set of places of $k$. We fix a central division $k$-algebra $D$, put $D_{\BA}=D\otimes_k \BA$ and set \[G=G_m(D_\BA).\] We denote by $l$ a quadratic extension of $k$. We suppose that there exists a $k$-embedding of $l \in \CM_m(D)$, which we fix. 
We denote by $\mathcal{H}$ the $l$-algebra centralizing $l$ in $\CM_m(D)$ and set \[H:=(\mathcal{H}\otimes_k\BA)^\times\leq G.\] The groups $G$ and $H$ are $\BA$-points of algebraic groups $\mathbf{G}$ and $\mathbf{H}$ defined over $k$ by definition. We put 
$H_{k}:=\mathbf{H}(k)$ and $G_k:=\mathbf{G}(k)$. By the analysis done in Section \ref{sec normalizer}, there exists an element $u\in G_k$ such that 
\[N_{G_k}(H_{k})=H_{k}\sqcup u.H_{k}.\]

In this paragraph we prove that the global sign of certain distinguished cuspidal automorphic representations of $G$ is equal to one. We recall that a smooth cuspidal automorphic representation $\Pi$ of $G$ (see \cite[Section 2.7]{BCZ}) with central character trivial on the center $Z_G$ of $G$ is called distinguished if the convergent period integral \[\phi\mapsto \int_{Z_G H_k\backslash H} \phi(h)dh\] does not vanish on $\Pi$. The following lemma 
follows from a simple change of variable in the period integral. 

\begin{lem}\label{lm trivial global character}
If $\Pi$ is a cuspidal automorphic representation of $G$, then the $H$-period on $\Pi$ is fixed by $u$.
\end{lem}

In order to compute the global sign, we also need to compute the root number of a distinguished cuspidal automorphic representation in the following special case. 

From now on, let $\JL$ denote the global Jacquet-Langlands correspondence defined in \cite{Badulescu-Global-JL} and \cite{Badulescu-Renard}.

\begin{lem}\label{lm global root nb}
Let $\Pi$ be a distinguished cuspidal automorphic representation of $G$ with cuspidal Jacquet-Langlands transfer. We suppose that $l/k$ splits at every Archimedean place of $k$, and that there are two places $v_1, \ v_2 \in \CP(k)$ such that $\Pi_{v_1}$ is cuspidal 
and $\JL_{v_2}(\Pi_{v_2})=\St_n(\triv_{v_2})$. Then $\JL(\Pi)$ has a Shalika period hence it is selfdual, and moreover the central value of the Godement-Jacquet $L$-function of $\JL(\Pi)$ is nonzero: \[L(1/2,\JL(\Pi))\neq 0.\] In particular its Godement-Jacquet root number is trivial: \[\epsilon(\JL(\Pi))=1.\]  
\end{lem}
\begin{proof}
Our local assumptions imply that the assumptions of \cite[Theorem 1.4]{XZ} are satisfied, hence $\JL(\Pi)$ has a Shalika period so it is selfdual. Then noting that standard $L$ functions of cuspidal automorphic representations are entire by \cite{GJ}, it follows again from 
\cite[Theorem 1.4]{XZ} that $L(1/2,\JL(\Pi))\neq 0$. Finally the functional equation of the standard $L$ function of $\JL(\Pi)$ (see \cite{GJ}) together with the fact hat $\JL(\Pi)$ is selfdual imply that $\epsilon(\JL(\Pi))=1$. 
\end{proof}

We obtain the sign of global linear periods as a corollary.

\begin{cor}\label{cor global sign}
Let $\Pi$ be a distinguished cuspidal automorphic representation of $G$ as in Lemma \ref{lm global root nb}. Then 
\[\prod_{v\in \CP(k)} \sg(\Pi_v)=1.\]
\end{cor}
\begin{proof}
First the product is well defined thanks to Proposition \ref{prop UR sign}. Moreover the period integral is factorizable thanks to Theorem \ref{thm mult 1} and Theorem \ref{thm mult 1 Arch} for Archimedean linear periods. The result now follows from, Lemmas 
\ref{lm trivial global character} and \ref{lm global root nb}.
\end{proof}

\subsection{Application to the sign of cuspidal representations}

The determination of the sign in this case follows again from a globalization argument, and the obvious computation of the sign of the trivial representation. For this latter part, we recall that the local root number of 
$\1_{D^\times}$ is equal to $-1$ if $d$ is even and $1$ if $d$ is odd. This immediately implies:

\begin{lem}\label{lm trivial sign}
Suppose that $d$ is even. Then one has $\sg(\1_{D^\times})= -1 $.
\end{lem}

Here we compute the sign of linear periods, using Lemma \ref{lm global root nb} and a globalization argument of Prasad and Schulze-Pillot. 

\begin{thm}\label{thm cusp}
Let $\pi$ be a distinguished cuspidal representation of $\GL_m(D)$. Then \[\sg(\pi)=(-1)^m.\] 
\end{thm}
\begin{proof}
Let $\frac{x}{d}+\BZ$ be the Hasse invariant of $\CM_m(D)$ where we take $x\in [1,d]$ and coprime to $d$, and define $a$ by the equality 
$\frac{a}{n}=\frac{x}{d}$. 

We start with the case of twisted linear periods. By Krasner's lemma and the weak approximation lemma, one can find a number field $k$ and a quadratic extension $l/k$ such that $l_v/k_v=E/F$ for some finite place $v$ of $k$, $l/k$ remains unsplit at every place of a set $S_0$ of finite places of $k$ of size $|S_0|=a$ which does not contain $v$, but $l/k$ splits at every Archimedean place of $k$ (see for instance \cite[Section 9.6]{AKMSS} for the details). Then by the Brauer-Hasse-Noether theorem, one can also find a division algebra $\CD$ with center $k$ such that $\CD_v=\CM_m(D)$, $\CD_{w}$ is a central division algebra $\mathfrak{D}_w$ over $k_w$ of Hasse invariant $\frac{-1}{n}+\BZ$ for each $w$ in $S_0$, and which is split at every other place of $k$. In this situation we observe that $l$ automatically embeds as a $k$-subalgebra of $\CD$ thanks to \cite[Theorem 1.1]{SYY}. Finally fix some finite $v_1$ outside $\{v\} \cup S_0\cup S_\infty$ such that $l_{v_1}/k_{v_1}$ is split. Putting  
\[S:=\{v\} \cup 
\{v_1\}\cup S_0\cup S_\infty\] with $S_\infty$ the set of Archimedean places of $k$, by \cite[Theorem 4.1]{PSP}, there exists a cuspidal automorphic representation $\Pi$ of $(\CD\otimes_k \BA)^{\times}$ which is unramified outside $S$, such that $\Pi_v=\pi$, and $\Pi_w=\triv_{\mathfrak{D}_w^\times}$ for $w$ in $S_0$, and $\Pi_{v_1}$ is cuspidal (and distinguished). Note that $\JL_w(\pi_w)=\St_n(\triv_w)$ for all $w\in S$ hence Lemma \ref{lm global root nb} applies. Note also that $\JL_{v_1}(\Pi_{v_1})=\Pi_{v_1}$ because $\CD_{v_1}$ is split. Now the combination of Proposition \ref{prop UR sign}, Proposition \ref{prop generic linear sign}, Lemma \ref{lm trivial sign} and Corollary \ref{cor global sign}, we obtain the formula:
\[\sg(\Pi_v)=(-1)^{a}.\] The result now follows from the fact that if $m$ is even, then $a=mx$ as well, whereas if $m$ is odd then $d$ must be even, $a=mx$ is also odd because $x$ and $d$ are coprime.

For linear periods, we use again Krasner's lemma and the weak approximation lemma, but this time to find a quadratic extension $l/k$ such that 
$l_v/k_v\simeq (F\times F)/F$ for some finite place $v$ of $k$, $l/k$ remains unsplit at every place of a set $S_0$ of finite places of $k$ of size $|S_0|=a$ which does not contain $v$, and again $l/k$ is split at every Archimedean place. The rest of the proof is the same (we can however observe that in this situation $m$ is automatically even, so \cite[Theorem 1.1]{SYY} applies again).
\end{proof}

\begin{rem}\label{rem poisson}
When $G$ is split (and also when $G$ is split over a quadratic extension), one can give a local proof of the above result, extending that of Prasad for inner forms of $\GL_2$, and which applies directly to the class of $H$-cuspidal representations of $G$ (see \cite{KT}). By \cite{KT} a cuspidal representation of $G$ is automatically $H$-cuspidal. We very briefly sketch the idea, and mention that the proof uses the realization of the contragredient of an irreducible representation using transpose inverse, hence the restrictive hypothesis on $G$. One just needs to modify an argument from the thesis \cite{Ok} of Youngbin Ok based on the Poisson summation formula, which itself extends to $\GL_n$ the quick proof by Deligne \cite{Del} of a result of Fr\"ohlich and Queyrut \cite{FQ} for $\GL_1$. Ok uses the the Poisson summation formula together with an extension of the Godement-Jacquet functional equation for relative matrix coefficients to prove the triviality of local root numbers of cuspidal representations distinguished by a Galois involution. The method of Ok in fact proves the result for relatively cuspidal representations as has been observed by Offen in \cite{Offen}. In fact \cite[Sections 2.2 and 2.3]{Offen} apply with minor modifications to our setting, using the same results from \cite{GJ}. The only difference is that in the Poisson summation formula \cite[Lemma 2.2]{Offen}, the orthogonal of the Lie algebra $\mathcal{H}$ of $H$ is $u.\mathcal{H}$ instead of $\mathcal{H}$ itself. 
Following the proof of \cite[Proposition 2.1]{Offen}, this will claim that the gamma factor of an $H$-cuspidal $\pi$ at $s=1/2$ is equal to $(-1)^{m(d-1)}\chi_{\pi}(u)=(-1)^{m}\chi_{\pi}(u)$, but this gamma factor is equal to the root number as $\pi$ is self dual and because $L(1/2,\pi)\neq 0$ because $\pi$ is unitary. We refer to the proof of Theorem \ref{thm Arch sign} for more details when $\pi$ is a discrete series representation and $G$ is a real Lie group. 
\end{rem}

\begin{rem}
Using Remark \ref{rem poisson}, one can remove the assumption that $l_{v_1}/k_{v_1}$ is split in the proof of Theorem \ref{thm cusp}.
\end{rem}

\begin{rem}\label{rem poiscaille}
When $F$ is a finite field, the sign problem makes sense as well. Multiplicity at most one is known for linear periods, at least on cuspidal representations (\cite{Sgalois}), and can certainly be checked for twisted linear periods following \cite{G}. In this setting the computation alluded to in \ref{rem poisson} works equally well and relates the sign to the Godement-Jacquet gamma factors defined in \cite{Rgamma}, with the extra complication that $G$ has not total measure in its Lie algebra anymore. Hence one has to use cuspidality to prove vanishing of some terms, in order to use the Poisson summation formula within the Godement-Jacquet functional equation. The details will appear elsewhere, and will be consistent with our result for level zero cuspidal representations.
\end{rem}

\section{The sign of generic representations}\label{sec main}

As an immediate corollary of Theorem \ref{thm linear std}, Proposition \ref{prop closed sign}, Proposition \ref{prop open sign}, and Theorem \ref{thm disc}, we achieve the goal of this paper when $F$ is $p$-adic. For Archimedean $F$ we refer to the next section. 

\begin{thm}\label{thm gen}
Let $F$ be any local field of characteristic zero, and $\pi$ be a distinguished generic representation of $\GL_m(D)$ Then \[\sg(\pi)=(-1)^m.\] 
\end{thm}

\section{The case of Archimedean fields}\label{sec Arch}

Here $D$ is either $\BC$, $\BR$ or the Hamilton quaternions $\BH$. This case is actually simpler up to the classification of distinguished generic representations, which is more difficult. First we claim that the strategy for reducing to discrete series representations is exactly the same. Indeed, the classification theorem \ref{thm linear std} is proved by Suzuki and Tamori in \cite{ST} for twisted linear periods, and in Appendix \ref{app ST} for linear periods, together with \cite[Theorem 5.4]{MOY} for the sufficiency of the condition obtained in these references. The reduction then relies on open and closed intertwining periods and the properties \eqref{x} and \eqref{y} of epsilon factors before Section \ref{sec lin per}, which as we already observed are also valid in the Archimedean setting. We refer to \cite{MOY} for the state of the art concerning Archimedean intertwining periods, however we mention that in the case of interest to us, the theory of Archimedean open intertwining periods is fully established by the works of Delorme-Brylinski \cite{Brylinski-Delorme-H-inv-form-MeroExtension-Invention} and Carmona-Delorme \cite{Carmona-Delorme-H-inv-form-FE-JFA}. Now there is not too much to say in the case of discrete series representations when $D=\BC^\times$ as they are characters of $\BC^\times$. The sign of discrete series representations of $\GL_2(\BR)$ admitting a linear model is covered by the results of Friedberg and Jacquet in Section \ref{sec gen sign}, and is actually nothing else than the $\GL_2(\BR)\times \GL_1(\BR)$ Rankin-Selberg equation at $s=1/2$.

When $D=\BH^\times$ then discrete series representations of $\BH^\times$ are finite dimensional and the only possible distinction is with respect to the twisted linear model. We treat this case together with that of twisted linear models for 
discrete series representations of $\GL_2(\BR)$. In such cases Prasad's computation of the sign of cuspidal representations using the Godement-Jacquet functional equation in \cite[Theorem 4]{P} when $F$ is $p$-adic applies up to some extra observations. 

\begin{thm}\label{thm Arch sign}
Let $\pi$ be a distinguished discrete series representations of $\GL_m(D)$, then $\sg(\pi)=(-1)^m$. 
\end{thm} 
\begin{proof}
Let us focus on the case of twisted linear models for $G=\BH^\times$ as well as $G=\GL(2,\BR)$, as we already discussed all other cases before the theorem. We treat both cases together though the arguments can be simplified a lot when  $G=\BH^\times$. First, $\pi$ is trivial on the central subgroup $\BR_{>0}$. Moreover, by density of the space $\pi_0$ of $\mathrm{SO}_2(\BR)$-finite vectors, the continuous and nonzero $\BC^\times$-invariant linear $\ell$ form on $\pi$ restricts as a non trivial linear form $\ell_0$ to that space. Hence it must restricts non trivially to the subspace of $\mathrm{SO}_2(\BR)$-fixed vectors, and in particular there is $v_0\in \pi$ which is fixed under $\BC^\times$ such that $\ell(v_0)\neq 0$. Of course the space of $\BC^\times$-fixed vectors is a line in $\pi$, otherwise the space $\BC^\times$-invariant linear forms on $\pi_0$ would have dimension $2$, and thus that of continuous $\BC^\times$-invariant linear forms on $\pi$ as well according to the automatic continuity \cite[Theorem 1]{Brylinski-Delorme-H-inv-form-MeroExtension-Invention}, so the sign of $\pi$ is actually detected by the action of $N$ on $\BC\cdot v_0$. Let $\langle \ , \ \rangle$ be a $G$-invariant scalar product on $\pi$ and call $\pi_1$ the completion of $\pi$ with respect to it (so that $\pi$ is the space of smooth vectors in $\pi_1$). We recall that this scalar product identifies $\overline{\pi_1}$ and $\pi_1^\vee=\pi_1$, and that the same holds without the index $1$ as well as with the index $0$. In particular through this identifications, $\ell_0$ identifies with $v\to \langle v , v_0 \rangle$, and $c_0(g):=\langle \pi(g)v_0, v_0  \rangle$ is a bi-$\BC^\times$-invariant matrix coefficient of $\pi_0$. The Godement-Jacquet functional equation \cite[Theorem 8.7]{GJ} for $s=1/2$ now looks like 
\[\int_G  \widehat{\Phi}(g)c_0(g^{-1}) \nu(g)dg=(-1)^{m(d-1)}\e(\phi_\pi)\int_G  \Phi(g)c_0(g)\nu(g) dg\] where $\Phi$ can be any Schwartz function on the Lie algebra $\mathfrak{g}$ of $G$, the Fourier transform $\widehat{\Phi}$ is defined with the character $\Psi$ of $\mathfrak{g}$ as in \cite[p. 113]{GJ}, and where we canceled out the nonzero L value $L(1/2,\pi_0)=L(1/2,\pi_0^\vee)$ on both sides of the functional equation. Moreover, both sides of the functional equation are actually absolutely convergent by \cite[(9.18) on p. 132]{GJ}. Now let $u$ be our representative in the non trivial class of $N_G(H)$. We observe that the orthogonal $\mathfrak{h}^\perp$ of the Lie algebra $\mathfrak{h}$ of $H$ with respect to $\Psi$ is equal to $u\cdot\mathfrak{h}=\mathfrak{h}\cdot u$. Rewriting the absolutely convergent LHS of the functional equation, we obtain:
\[\int_G  \widehat{\Phi}(g)c_0(g^{-1}) \nu(g)dg=\int_{G/H} (\int_H \widehat{\Phi}(hg) \nu(hg) dh) c_0(g^{-1})dg\]
\[=\int_{G/H} (\int_{\mathfrak{h}} \widehat{\Phi}(xg) \nu(g) dx) c_0(g^{-1})dg\] where $dx$ is the normalized Haar measure on $\mathfrak{h}$. Because \[\widehat{\Phi}(xg)=\nu(g)^{-2}\widehat{\rho(g^{-1})\Phi}(x)\] where $\rho$ denotes right translate, by the Poisson formula we obtain  
\[\int_{\mathfrak{h}} \widehat{\Phi}(xg)  \nu(g)dx=\int_{\mathfrak{h}}\Phi(xg^{-1}) \nu(g)^{-1} dx,\] hence 
\[\int_G  \widehat{\Phi}(g)c_0(g^{-1}) \nu(g)dg= 
\int_{G/H} (\int_{u\cdot\mathfrak{h}}\Phi(xg^{-1}) \nu(g)^{-1} dx) c_0(g^{-1})dg\]
\[=\int_{G/H} (\int_{H}\Phi(u^{-1}hg^{-1}) \nu(u^{-1}hg^{-1}) dh) c_0(g^{-1})dg\]
\[=\nu(u^{-1})\int_G \Phi(u^{-1}g^{-1})  c_0(g^{-1})\nu(g^{-1}) dg\]
\[=\int_G \Phi(g)  c_0(ug)\nu(g)dg=\chi_{\pi}(u)\int_G \Phi(g) \nu(g)c_0(g)dg\]
Certainly the Schwartz function $\Phi$ can be chosen such that the Zeta integral $\int_G \Phi(g) \nu(g)c_0(g)dg$ is nonzero, so we deduce from the functional equation that 
\[\chi_{\pi}(u)=(-1)^{m(d-1)}\e(\phi_{\pi})=(-1)^m\e(\phi_{\pi}).\]
\end{proof}

\appendix
\section{multiplicity one for linear periods of inner forms}\label{app mult 1}

In this appendix $F$ is allowed to be any local field of characteristic $0$. We recall that $D$ is a central $F$-division. Let $l,l'$ and $m$ be positive integers such that $l + l' =m$. We prove here that the symmetric pair \[(G:=\GL_m(D),H:=\GL_l(D) \times \GL_{l'}(D))\] is a \textit{Gelfand pair} when $|l-l'|\leq 1$, i.e., that if $\pi$ is an irreducible representation of $\GL_m(D)$ (of Casselman-Wallach type when $F$ is Archimedean), then 
\[\dim(\Hom_{H}(\pi,\BC))\leq 1.\] When $D=F$ is non Archimedean, this result is due to Jacquet and Rallis (\cite{JR}). Later Aizenbud and Gourevitch gave another proof, inspired by that of Jacquet and Rallis, which applies to all local fields of characteristic zero in 
\cite[Theorem 8.2.4]{Aizenbud-Gourevitch}. This proof was extended when $l=l'$ to unitary representations of inner forms by Chong Zhang, who obtained the following statement as a key step in the process (see \cite[Proposition 4.3 and Remark 4.7]{Zlinear}), but we claim that this step is valid for any $l$ and $l'$:

\begin{prop}
Let $\pi$ be an irreducible representation of $G$. Then \[\dim(\Hom_H(\pi,\BC))\dim(\Hom_H(\pi^\vee,\BC))\leq 1.\]
\end{prop}
\begin{proof}
Following the proof of \cite[Proposition 4.1]{Zlinear} and using the terminology introduced in \cite{BM}, one sees that the descendants of the pair $(G,H)$ can only be of diagonal type, of Galois type or of the form \[(R_{L /F}\GL_{a+b}(D' ), R_{L /F}\GL_a(D' )\times R_{L /F}\GL_b(D' ))\] for $D'$ an $L$-division algebra, where $L$ is a finite extension of $F$ and $R_{L /F}$ is Weil restriction of scalars. This is enough to conclude according to the discussion before \cite[Proposition 4.3]{Zlinear} and the references to \cite{Aizenbud-Gourevitch} given there.
\end{proof}

In order to prove that $(G,H)$ is a Gelfand pair when $|l-l'|\leq 1$, it is enough to prove that $H$-distinguished irreducible representations of $G$ are selfdual. This has been proved in Corollary \ref{cor symp gen} in the $p$-adic case, and follows from a verbatim adaptation of the proof of \cite[Theorem 6.7]{BM} in the Archimedean case. Hence we obtain:

\begin{thm}\label{thm mult 1 Arch}
Let $\pi$ be an irreducible representation of $G$, and suppose that $|l-l'|\leq 1$. Then \[\dim(\Hom_{H}(\pi,\BC))\leq 1.\]
\end{thm}

When $\pi$ is unitary, we can remove the assumption on $p$ and $q$ by the observation in the proof of \cite[Corollary 4.4]{Zlinear}.

\section{Bernstein-Zelevinsky theory for the mirabolic subgroup of inner forms and linear periods}\label{app BZ der}

Here we only make observations without proofs, about some results which remain valid from the mirabolic restriction theory of Bernstein-Zelevinsky (\cite{BZ1} and \cite{BZ2}) for inner forms of the generalized linear group. We define the mirabolic subgroup $P_m$ of $G_m:=\GL_m(D)$ as the group of matrices in $G_m$ with bottom row equal to $(0,\dots,0,1)$. We denote by $U_m$ its unipotent radical, so that the mirabolic subgroup is the semi-direct product $P_m=G_{m-1}.U_m$, where we identify $G_{m-1}$ with a subgroup of $G_m$ by the embedding 
\[g\mapsto \diag(g,1).\] The group $G_{m-1}$ thus acts naturally on the group $\widehat{U_m}$ of characters of $U_m$ by 
\[(g,\psi)\mapsto \psi(g^{-1}\ \cdot \ g),\] with exactly two orbits: the trivial character and the others. Moreover all non degenerate characters (see \cite{BH}) of the maximal unipotent upper triangular subgroup of $G_m$ are conjugate under the action of the diagonal matrices in $G_m$. This makes some part of the machinery of Bernstein-Zelevinsky work the same and we simply observe that the following results hold with exactly the same proofs. First the functors $\Phi^\pm$ and $\Psi^\pm$ (and $\hat{\Phi}^+$ defined as in \cite[Section 3]{BZ2}, satisfy the statements of \cite[Propositions 3.2 and 3.4]{BZ2} with the same proofs (relying on \cite[Section 5.11]{BZ1}). Following the authors of 
\cite{BZ2}, for 
$\tau$ a smooth representation of $P_m$ and $k$ between $1$ and $n$ we put \[\tau^{(k)}=\Psi^-(\Phi^-)^{k-1}(\tau),\] and 
 \[\tau^{(0)}=\tau.\]
Then the Bernstein-Zelevinsky filtration still exists for smooth representations of $P_m$:

\begin{thm}{\cite[Corollary 3.4]{BZ2}}
Let $\tau$ be a smooth representation of $P_m$ with $m\geq 2$. Then $\tau$ admits a filtration \[\{0\}\subseteq \tau_m \subseteq \dots \subseteq \tau_1=\tau \] 
with \[\tau_k/\tau_{k+1}\simeq (\Phi^+)^{k-1}\Psi^+(\tau^{(k)}).\]

\end{thm}

Following again Bernstein and Zelevinsky, we define the derivatives of a representation of $G_m$ to be that of its restriction to $G_m$. Now let $\pi$ be a cuspidal representation of $G_m$. Because of the filtration above, and because $\pi^{(k)}=\{0\}$ for all $k=1,\dots,m-1$ (this is by the definition of the derivatives, which factor through some Jacquet module), we obtain the following corollary of the Bernstein-Zelevinsky filtration above:

\begin{cor}
Let $\pi$ be a cuspidal representation of $G_m$. Then $\pi=(\Phi^{+})^{m-1}(\Psi^+)(\pi^{(n)})$ as a representation of $P_m$, where $\pi^{(n)}$ is a finite dimensional $D$-vector space, the dimension of which is that of the generalized Whittaker model of $\pi$ in the sense of \cite{BH}.
\end{cor}
\begin{proof}
The only assertion remaining is that on the dimension of $\pi^{(n)}$. By definition its dual is the space of generalized Whittaker functionals on $\pi$, and it is finite dimensional by \cite{MW}.
\end{proof}

In turn, this is sufficient for the following result concerning cuspidal representations of $G_m$ to hold with the same proof as in \cite{MatCRAS} (see also \cite{MatCMRL} for the version with the twisting character).

\begin{prop}\label{prop cusp p=q}
Let $p$ and $q$ be two non negative integers such that $p+q=m\geq 2$. Let $\mu$ be a character of $H_{l,l'}$ and $\pi$ be a cuspidal representation of $G_m$. If 
 $\pi$ is $(H_{l,l'},\mu)$-distinguished, then $l=l'$.
\end{prop}

We extend this result to discrete series representations in Section \ref{sec steinberg}. Here we treat separately the case of 
Steinberg representations, as its proof is different and uses the theory of derivatives. 

More precisely, a careful analysis of \cite{BZ2} and \cite{Z} shows that the full theory developed in these papers hold for representations with cuspidal support consisting of characters of $G_1=D^\times$. 
Let us explain the crucial point. The result of \cite[Theorem 4.2]{BZ2} for a representation $\pi:=\chi_1\times \dots \times \chi_r$ induced by characters of $D^\times$ becomes in our setting that $\pi$ is irreducible as soon as $\chi_i\chi_j^{-1}\neq \mu$ for any $i\neq j$. But following \cite{BZ2}, with notations and terminology identical to that of \cite[Section 4.8]{BZ2}, the natural 1-pairing between $\pi$ and $\tilde{\pi}$ induces a $\mu$-pairing between $\pi$ and $\overline{\pi}:=\mu \pi$, which induces by restriction to the mirabolic subgroup $P_r$ a $\Delta$-pairing between $\pi_{|P_r}$ and 
$\overline{\pi}_{|P_r}$. This latter assertion becomes false if one replaces $\mu$ by $\mu^s$ for any other real number $s$, but for $s=1$ it implies that the proof of \cite[Lemma 4.7]{BZ2} remains valid when the cuspidal representations $\rho_i$ there are replaced here by characters $\chi_i$ of $D^\times$. The claim one has to check in turn is that this lemma is the key to have a complete theory of derivatives. In particular the following proposition holds, as one can check from its proof in \cite[Proposition 9.5]{Z} and the formula for Jacquet modules of generalized Steinberg representations in \cite[Proposition 3.1]{Tglna}. 

\begin{prop}
Let $m\geq 1$ and $k$ be between $0$ and $n$, and let $\chi$ be a character of $G_1$. Then 
\[\St_n(\chi)^{(k)}=\nu^{k/2}\St_{n-k}(\chi).\]
\end{prop}

The the following follows immediately as in \cite[Theorem 3.1]{MatJNT}

\begin{prop}\label{prop st p=q}
Suppose that the representation $\St_n(\chi)$ above is $(H_{l,l'},\mu)$-distinguished for some character $\mu$ of $H_{l,l'}$. 
Then $p=q$.
\end{prop}

\section{A multiplicity zero result for generic representations}\label{app mult zero}

Here $F$ is $p$-adic. He prove a multiplicity zero result which applies to a class larger than that of generic representations. 

\begin{lem}\label{lm prod}
Let $p\geq q$ be two non negative integers, $m_0$ and $m_1$ be two positive integers such that $m_0+m_1=l+l':=m$. Set $G=G_m$, and set $H=H_{l,l'}$. Let $\pi_0$ be an irreducible representation of $G_{m_0}$ with cuspidal support containing no character of $G_1$, let $\pi_1$ an irreducible representation of $G_{m_1}$ with cuspidal support consisting of characters of $G_1$ only, and let $\chi$ be a character of $H$. In such a situation, if 
$\pi_0\times \pi_1$ is $\chi$-distinguished, then $m_0$ is even, $q\geq m_0/2$, and there exist characters $\chi_0$ of $H_{m_0/2 ,m_0/2 }$ and $\chi_1$ of $H_{l-m_0/2 ,l'-m_0/2}$ such that $\pi_0$ is $\chi_0$-distinguished and $\pi_1$ is $\chi_1$-distinguished.
\end{lem}
\begin{proof}
Once again this follows from the geometric lemma as in \cite{OJNT}. Realize $\pi_0$ and $\pi_1$ as Langlands quotients of the standard modules $S_0$ and $S_1$ respectively. Any $\chi$-invariant on $\pi_0\times \pi_1$ is the descent of a $\chi$-invariant linear form on $S_0\times S_1$. Set $P=MU$ be the standard parabolic subgroup of $G$ of type $(m_0,m_1)$. Reasonning as in the beginning of the proof of Theorem \ref{thm linear std}, and by our assumption, no duality relation can occur between the subsegments of the product $S_0$ and those of the product $S_1$. This already implies that the double cosets $PxH$ possibly contributing to distinction of $\pi_0\times \pi_1$ are such that the representative $x$ belongs to the standard Levi subgroup $M$. Moreover if one writes 
$x=\diag(x_0,x_1)$, the element $x_0$ can only have even $m_{i,i}=2m_{i,i}^+$ in its associated subpartiation by Theorem \ref{thm p=q}. The result now follows from the second equality in the proof of \cite[Proposition 3.1]{OJNT}.
\end{proof}

As a consequence of the discussion in Appendix \ref{app BZ der} about the theory of derivatives for representations the cuspidal support of which consist of characters of $G_1$, we also obtain the following multiplicity zero result. Its proof is the same as that of the erratum to \cite[Theorem 3.2]{MatJNT}, which is available on the author's webpage but not published. We thus reproduce the argument here, and recall that if $\Delta$ is a cuspidal segment, then $r(\Delta)$ denotes the cuspidal representation forming its right end. 

\begin{lem}\label{lm anti}
Let $\D_1,\dots,\D_r$ be cuspidal segments with supports consisting of characters of $G_1$, and suppose that they are right anti-ordered, i.e. $e(r(\D_{i+1}))\geq e(r(\D_i))$ for $i=1,\dots, r-1$. Then the representation $\pi=\d(\D_1)\times \dots \times \d(\D_r)$ of $G$ cannot by $H_{l,l'}$ distinguished if $|l-l'|\geq 2$. 
\end{lem}
\begin{proof}
We observe that for any non negative integer $n_i$, the segments \[\D_1^{(n_1)},\dots,\D_r^{(n_1)}\] are still right anti-ordered. Then one replaces the class of representations to which the induction is applied in the proof of \cite[Theorem 3.2]{MatJNT} (defined by a non preceding condition which is not preserved by taking derivatives of each segments as claimed in the proof of \cite[Theorem 3.2]{MatJNT}) by that of products as in the statement of the lemma. Then the proof of \cite[Theorem 3.2]{MatJNT} applies without any modification, except the initial steps of the induction which are for $n=2$ and $3$. For these cases we see that for $G_2$ and $G_3$, the only product of discrete series representations having a character as a quotient are respectively of the form $\xi\times \xi \nu^{-1}$ and $\xi\times \xi \nu^{-1}\times \xi \nu^{-2}$, and the corresponding segments are visibly not right anti-ordered.
\end{proof}

As an immediate consequence of Lemmas \ref{lm prod} and \ref{lm anti}, we obtain:

\begin{thm}\label{thm mult 0}
Let $\pi$ be an irreducible representation of $G$ which can be written as a commutative product of discrete series representations, for example a generic representation of $G$. Let $H=H_{l,l'}$ with $|l-l'|\geq 2$, and $\chi$ be a character of $H$. Then $\pi$ cannot be $\chi$-distinguished. 
\end{thm}

\section{Archimedean standard modules with a linear period}\label{app ST}

\begin{center}\textit{By Miyu Suzuki and Hiroyoshi Tamori}\end{center}

We prove a necessary condition for a standard module of $\GL_m(D)$ has non-zero linear period.
Here,  $D=\BR$,  $\BC$ or $\BH$.
We follow the same line as \cite{ST},  in which the authors treated the case of twisted linear period.

\subsection{Notation}\label{sec not ST}
\label{sec:notation}
We recall the notations that will be used here. Let $D\in\{\BR,  \BC,  \BH\}$,  where $\BH$ is the quaternion division algebra and $m$ be a positive integer.
Set
    \[
    G=G_m=\GL_m(D),  \qquad l=
        \begin{cases}
        \frac12 m & \text{if $m$ is even,} \\
        \frac12(m-1) & \text{if $m$ is odd.}
        \end{cases}
    \]

Put $\delta=\diag(1,  -1,  \ldots,  (-1)^{m-1})\in G$ and write $\theta=\Ad(\delta)$ for the inner automorphism $x\mapsto \theta(x)=\delta x\delta^{-1}$ of $G$.    
The group of fixed points $H_m=H:=G^\theta$ is a symmetric subgroup and $H_m=H_{l,l}$ when $m$ is even, whereas $H_m=H_{l+1,l}$ when $m$ is odd. 

A parabolic subgroup and a Levi subgroup of $G$ are assumed to be blockwise upper triangular and blockwise diagonal.  
Let $(m_1,  \ldots,  m_r)$ be a partition of $m$ and $P$ a parbolic subgroup with a Levi subgroup 
consisting of blockwise diagonal matrices of the form $\diag(x_1,\ldots,x_r)$ with $x_i\in \GL_{m_i}(D)$.
When $D\neq \BR$,  we assume that $r=m$ and $m_1=\cdots=m_r=1$,  \emph{i.e.} $P$ is the minimal parabolic subgroup.
When $D=\BR$,  we assume that $m_i\in\{1,  2\}$ for all $i$.
These conditions are equivalent to that Levi subgroups of $P$ have essentially square integrable representations,  and such parabolic subgroups are called \emph{cuspidal}.

\subsection{$H_m$-orbits in $P\backslash G$}\label{sec orb ST}
Let $P$ be a cuspidal parabolic subgroup of $G$ corresponding to the partition $(m_1,  \ldots,  m_r)$.
From \cite[Section 3.1]{MatCRELLE},  we recall the study on a set of representatives of the $H_m$-orbits in the flag variety $P\backslash G$.
Note that \cite{MatCRELLE} treats the $p$-adic case but the same holds in our setting.

Let $V$ be a right $D$-vector space of dimension $m$ and $e_1,  \ldots,  e_m$ a basis of $V$.
We regard $G$ as the group of $D$-linear automorphisms of $V$ which acts on $V$ from the left.
In particular,  note that $\delta\in G$ defines an involution on $V$.
Set $V_i^\circ=\mathrm{Span}_D\{e_i \mid 1\leq i \leq m_1+\cdots+m_i\}$ for each $i=1,  \ldots,  r$.
We obtain a flag $V_0^\circ=0\subset V_1^\circ \subset \cdots \subset V_r^\circ=V$ such that $\dim_D(V_i^\circ)=m_1+\cdots+m_i$.

Let $\CF_P(V)$ be the set of flags $V_0=0\subset V_1 \subset \cdots \subset V_r=V$ of length $r$ such that $\dim_D(V_i)=m_1+\cdots+m_i$ for each $i$.
Then we obtain a bijection $P\backslash G\to \CF_P(V)$ by sending $g\in P\backslash G$ to the flag $0\subset g^{-1}V_1^\circ \subset \cdots \subset g^{-1}V_r^\circ=V$.
Obviously this is a $G$-equivariant map with respect to the right $G$-actions on $P\backslash G$ and $\CF_P(V)$.

For a flag $0\subset V_1 \subset \cdots \subset V_r=V$ in $\CF_P(V)$,  take a complement $W_{ij}\subset V_i\cap \delta(V_j)$ to $V_i\cap\delta (V_{j-1})+V_{i-1}\cap\delta(V_j)$ satisfying $W_{ji}=\delta(W_{ij})$ for each $i,  j=1,  \ldots,  r$.
Note that in particular $W_{ii}$ is $\delta$-stable.
Such family of subspaces $(W_{ij})$ exists since $\delta$ maps $V_i\cap\delta(V_{j-1})+V_{i-1}\cap\delta(V_j)$ to $V_j\cap\delta(V_{i-1})+V_{j-1}\cap\delta(V_i)$.

For each $i,  j$,  set $m_{ij}=\dim_D(W_{ij})$.
Let $W_{ii}=W_{ii}^+\oplus W_{ii}^-$ be the decomposition of $W_{ii}$ into the $(\pm1)$-eigensubspaces of $\delta$ and set $m_{ii}^\pm=\dim_D(W_{ii}^\pm)$.
Then we obtain an $r\times r$-symmetric matrix $(m_{ij})$ of non-negative integer entries attached to a flag $0\subset V_1 \subset \cdots \subset V_r=V$ in $\CF_P(V)$.
It satisfies for each $i,  j$ that  \setlength{\leftmargini}{25pt}
\begin{itemize}
\item[(P1):]   $m_i=\sum_{j=1}^r m_{ij}$,  \emph{i.e.} $m=\sum_{i, j=1}^r m_{ij}$ is a subpartition of $m=\sum_{i=1}^r m_i$.
\item[(P2):] $m_{ji}=m_{ij}$ and $m_{ii}=m_{ii}^+ +m_{ii}^-$.
\item[(P3):] $\sum_{i=1}^rd_i=[m]_2$, where $d_i:=m_{ii}^+-m_{ii}^-\in\BZ$ and $[m]_2\in\{0,1\}$ denotes the remainder when $m$ is divided by $2$.
\end{itemize}
Let $I_P(m)$ be the set of 
pairs $((m_{ij})_{ij}, (m_{ii}^{\pm})_i)$ satisfying $(\mathrm{P}1)$-$(\mathrm{P}3)$ where $(m_{ij})$ is a $r\times r$-matrix with $m_{ii}=m_{ii}^+ + m_{ii}^-$.

It is easy to see that two flags are in a same $H_m$-orbit in $\CF_P(V)$ if and only if they correspond to the same elements in $I_P(n)$ (see \cite[Proposition 3.1]{MatCRELLE}, the proof of which holds in the Archimedean setting).
Let $\CF_P(V)/H_m$ denote the set of $H_m$-orbits in $\CF_P(V)$.

\begin{lem}
The map $\CF_P(V)/H_m \to I_P(m)$ obtained above is injective,  \emph{i.e.} the set $I_P(m)$ is a complete set of invariants of $H_m$-orbits in $\CF_P(V)$.
\end{lem}

On the other hand,  for each element $s\in I_P(m)$,  we can explicitly construct flags in $\CF_P(V)$ which corresponds to that invariant.
In particular,  the map $P\backslash G/H_m \to  \CF_P(V)/H_m \to I_P(m)$ is a bijection.
For $s=[(m_{ij}),  (m_{ii}=m_{ii}^+ + m_{ii}^-)]\in I_P(m)$,  we define a block monomial matrix $\delta_s$ as follows:
\begin{itemize}
\item[(i)] the block partition of $\delta_s$ is the subpartition of $m=m_1+\cdots+m_r$ attached to $(m_{ij})$,  \emph{i.e.} 
    \begin{align}\label{eq:partition}
    m & = m_{11} + m_{12} + \cdots+m_{1r} \\ \nonumber 
     &  + m_{21} + m_{22} + \cdots + m_{2r}  \\ \nonumber 
     & \phantom{+ m_{21} + m_{22}}  \vdots  \\ \nonumber 
     & + m_{r1} + m_{r2} + \cdots + m_{rr}. 
    \end{align}
\item[(ii)] each diagonal $m_{ii}\times m_{ii}$-block is $\diag(I_{m_{ii}^+},  -I_{m_{ii}^-})$ and the other diagonal blocks are all zero.
\item[(iii)] for each $1\leq i< j \leq r$,  both the $m_{ij}\times m_{ji}$-block and the $m_{ji}\times m_{ij}$-block are $I_{m_
{ij}}$ (note that $m_{ij}=m_{ji}$).
\item[(iv)] all the other blocks are zero.
\end{itemize}
Then we can take a representative $u_s\in G$ of the double coset in $P\backslash G/H_m$ corresponding to $s\in I_P(m)$ so that $\delta_s=u_s\delta u_s^{-1}$ (\cite[Proposition 3.2]{MatCRELLE}).

Since $\delta_s^2=I_m$,  the inner automorphism $\theta_s=\Ad(\delta_s)$ on $G$ is an involution.
The group of fixed points $G^{\theta_s}$ equals $u_sH_m u_s^{-1}$.

Let $P_s$ be the standard parabolic subgroup of $G$ corresponding to the partition \eqref{eq:partition}.
Note that we have $P_s\subset P$ and that $P_s^{\theta_s}=(P\cap \theta_sP)^{\theta_s}$.
Let $P=MN$ and $P_s=M_sN_s$ be the Levi decomposition of $P$ and $P_s$,  respectively.
Then we have
\begin{itemize}
\item $M\cong \GL_{m_1}(D)\times \cdots \times \GL_{m_r}(D)$ and $M_s\cong \prod_{j=1}^r\GL_{m_{1j}}(D) \times \cdots \times \prod_{j=1}^r \GL_{m_{rj}}(D)$.
\item $P'_{M,  s}=M\cap P_s=M\cap\theta_sP$ is a parabolic subgroup of $M$ with Levi decomposition $P'_{M,  s}=M_s(M\cap N_s)=(M\cap\theta_sM)(M\cap\theta_sN)$. 
\item the involution $\theta_s$ acts on $M_s$ by a permutation of blocks and and an element of $M_s^{\theta_s}$ is of the form
    \begin{equation}\label{eq:Ms_fixed}
    x=\diag(x_{11}^+,  x_{11}^-,  x_{12},  \ldots,  x_{1r},  x_{21},  \ldots,  x_{r,r-1},  x_{rr}^+,  x_{rr}^-),
    \end{equation}
where $x_{ij}\in\GL_{n_{ij}}(D)$,  $x_{ii}^\pm\in\GL_{x_{ii}^\pm}(D)$ and $x_{ij}=x_{ji}$ for all $i\neq j$.
\end{itemize}
The following proposition is \cite[Proposition 3.6]{MatCRELLE}.
 
\begin{prop}\label{prop}
For parabolic subgroups $P$,  $P_s$ of $G$ and $P'_{M,  s}$ of $M$,  we write their modulus characters by $\delta_P$,  $\delta_{P_s}$ and $\delta_{P'_{M,  s}}$.
\begin{itemize}
\item[(1)] For $x\in M_s^{\theta_s}$,  we have $\delta_P\delta_{P'_{M,  s}}(x)=\delta_{P_s}(x)$.
\item[(2)] We write $x\in M_s^{\theta_s}$ as in the form \eqref{eq:Ms_fixed}.
Then we have
    \[
    \delta^2_{P_s^{\theta_s}}\delta_{P_s}^{-1}(x) = \prod_{1\leq i < j \leq r}
    |\det(x_{ii}^+)|^{d_j}|\det(x_{ii}^-)|^{-d_j}
    |\det(x_{jj}^+)|^{-d_i}|\det(x_{jj}^{-})|^{d_i}.
    \]
\end{itemize}
\end{prop}

\subsection{The main result}
We assume that all representations of (almost) linear Nash groups in this Appendix are 
smooth Fr{\' e}chet representations of moderate growth. 
Lie algebras will be denoted by Latin capital letters, and their Lie algebras will be denoted by the corresponding lower case German letters. 
We put subscript $\BC$ to denote their complexifications. 
Given a linear $f$ map between Lie algebras, we use the same letter $f$ to denote the complexification of $f$. 
For a complex linear space $\fr_{\BC}$, we write $\fr_{\BC}^{\vee}=\Hom_{\BC}(\fr_{\BC},\BC)$.  
Given an involution $\sigma$ on a Lie group $G$, we use the same letter $\sigma$ for its differential on the Lie algebra $\fg$ of $G$. 
The trivial representation of a group $G$ is denoted by $\triv_{G}$.

Take $P, s, \theta_s$ as in Section \ref{sec orb ST}.
Fix a nondegenerate invariant bilinear form $B$ on $\fg$. 
Let $\fs$ the Cartan subalgebra of the real Lie algebra $\fg$ consisting of diagonal real matrices. 
Write $\Sigma=\{e_i-e_j\mid 1\le i,j\le m, i\neq j\}$ for the root system for $(\fg,\fs)$ and fix a positive system $\Sigma^+=\{e_i-e_j\mid 1\le i<j\le m\}$ of $\Sigma$. 
For $\a \in \Sigma$, we define $X_{\a}$ to be the unique element in $\fs$ satisfying $B(X,X_{\a})=\a(X)$ for any $X\in\fs$. 
Set $\fa:=\fz(\fm)\cap\fg^{-\theta_s}$, where $\fz(\fm)$ denotes the center of $\fm$. 
We regard $\fa_{\BC}^{\vee}$ as a subspace of $\fs_{\BC}^{\vee}$ via the orthogonal decomposition $\fs=\fa\oplus\{X\in\fs | B(X,Y)=0\text{ for any $Y\in\fa$}\}$. 

For a closed subgroup $R$ of $G$, we define 
    \[
    \Sigma_{R}=\{\a\in\Sigma | \fg_{\a}\cap\fr\neq \{0\}\}.
    \]

\begin{thm}\label{app}
Let $\delta_i$ be an irreducible essentially square integrable representation of $\GL_{m_i}(D)$ for any $1\le i\le r$. 
Define $\l\in\fa_{\BC}^{\vee}$ so that $\delta_1\hat{\otimes}\cdots\hat{\otimes}\delta_r(\exp(X))$ is a scalar multiple by $e^{\l(X)}$ for $X\in\fa$. 
Here $\hat{\otimes}$ denotes the completed tensor product of nuclear Fr\'echet spaces.
Assume that $\re\l(X_{\a})\ge 0$ for $\a\in\Sigma^+$ and 
that $\pi:=\delta_1\times\delta_2\times\cdots\times\delta_r$ is $H_m$-distinguished.
Then there exists an involutive permutation $\z\in S_r$ satisfying the following conditions:
    \begin{align}\label{per}
        \begin{cases}
        \text{$m$ is odd and $\delta_i\cong\triv_{\GL_1(D)}$}&\text{if $\z(i)=i$ and $m_i=1$,}\\
        \text{$\delta_i$ is $H_{2}$-distinguished}&\text{if $\z(i)=i$, $D=\BR$ and $m_i=2$,}\\
        \text{$m_{\z(i)}=m_i$ and $\delta_{\z(i)}\cong \delta_i^{\vee}$}&\text{if $\z(i)\neq i$.}
        \end{cases}
    \end{align}
Here $\delta_i^{\vee}$ denotes the contragredient representation of $\delta_i$. 
In particular,  if $m=\sum_{i=1}^rm_i$ is odd, then there exists $1\le i\le r$ such that $m_i=1$ and $\delta_i\cong \triv_{\GL_1(D)}$.
\end{thm}

The proof is similar to the one in \cite[Theorem 1.2]{ST}. 
We only sketch the parts which are word by word the same as in \cite{ST}, and explain the other parts in detail.

Let us consider the condition on $s\in I_P(m)$ that
\begin{align}\label{parity}
\begin{cases}
d_i=0&\text{ for $1\le i\le r-1$\quad if $m$ is even},\\
d_{\le i}:=\sum_{1\le j\le i}d_j\in\{0,1\}&\text{ for $1\le i\le r-1$\quad if $m$ is odd}
\end{cases}
\end{align}
and put 
    \[
    J_P(m):=\{s\in I_P(m) \mid  \text{$M$ is $\theta_s$-stable, and \eqref{parity} holds}\}.
    \]
Remark that $M$ is $\theta_s$-stable if and only if $\fm$ is $\theta_s$-stable. 

As in \cite[Theorem 5.8]{ST}, let us prove the following Lemma.
\begin{lem}\label{lem}
    \[
    \dim\Hom_{H_m}(\pi,\triv_{H_m})\le \sum_{s\in J_P(m)}\dim\Hom_{M^{\theta_s}}(\delta_1\hat{\otimes}\cdots\hat{\otimes}\delta_r,\triv_{M^{\theta_s}}).
    \]
\end{lem}
\begin{proof}
Fix $s\in I_P(m)$ and $k\in\BN$. 
Set $\Xi:=(\Sigma_N\cap\theta_s\Sigma_{M})\cup(\theta_s\Sigma_N\cap\Sigma_{M})\cup(\Sigma_N\cap\theta_s\Sigma_N)\subset \Sigma^+$. 
From \cite[Lemma 4.4]{ST}, there exists $X\in \fs^{\theta_s}\cap [\fg,\fg]$ 
commuting with $\fm_s=\fm\cap\theta_s\fm$ such that
\begin{enumerate}
\item\label{i} $X\in\sum_{\a\in\Xi}\BR_{>0}X_{\a}$, 
\item\label{ii} the eigenvalue of the adjoint action of $X$ on $(\fg/(\fg^{\theta_s}+\fp))_{\BC}^{\vee}$ are all positive, and
\item\label{iii} $\a(X)>0$ for any $\a\in\Xi$.
\end{enumerate}
In particular, $X$ is of the form 
    \[
    X=\diag(x_{11}I_{n_{11}},x_{12}I_{n_{12}},\ldots,x_{r,r-1}I_{n_{r,r-1}},x_{rr}I_{n_{rr}})
    \]
with $x_{ij}\in\BR$, $x_{ij}=x_{ji}$ for $1\le i,j\le n$, and $x_{ii}> x_{jj}$ for $i<j$.

Put $K$ to be the maximal compact subgroup $\{g\in G\mid {}^t\overline{g}^{-1}=g\}$ of $G$ (Remark that $\overline{\ \cdot\ }$ denotes the main involution when $D=\BH$). 
From Proposition \ref{prop}(1), $\d_P^{1/2}=\d_{P_s}^{1/2}\d_{P'_{M,  s}}^{-1/2}$. 
As in the proof of \cite[Theorem 5.8]{ST}, it suffices to prove that 
the eigenvalues of the action of $X$ on the $0$-th relative Lie algebra homology
    \begin{align}\label{zero}
    H_0((\fm\cap\theta_s\fn)_{\BC},(\delta_1\hat{\otimes}\cdots\hat{\otimes}\delta_r)_{(K\cap M^s)\text{-fin}})
    \otimes\Sym^k(\fg/(\fg^{\theta_s}+\fp))^{\vee}_{\BC}\otimes(\d_{P_s^{\theta_s}}^{-1}\d_{P_s}^{1/2})\d_{P'_{M,  s}}^{-1/2}
    \end{align}
are nonzero if $s\in I_P(m)\setminus J_P(m)$ or $k>0$.  
Here $(\delta_1\hat{\otimes}\cdots\hat{\otimes}\delta_r)_{(K\cap M^s)\text{-fin}}$ denotes the space of $(K\cap M^s)$-finite vectors in $\delta_1\hat{\otimes}\cdots\hat{\otimes}\delta_r$.
Note that we used $u_sH_mu_s^{-1}\cap P=(P\cap\theta_sP)^{\theta_s}=P_s^{\theta_s}$.
See \cite[Remark 5.1]{ST}.

Set $\fm_i:=\{\diag(x_1I_{n_1},\ldots,x_rI_{n_r})\in\fg\mid x_1,\ldots,x_r\in D, x_j=0 \text{ if $j\neq i$}\}$, $A:=\{i\in \{1,\ldots,r\}\mid \theta_s\fm_i\cap\fn\neq 0\}$. 
Remark that $M$ is $\theta_s$-stable if and only if $A=\emptyset$. 
Since any noncompact simple ideal in $\fm$ is $\sl_2(\BR)$ and $\fm\cap\theta_s\fn$ is contained in $\fp_{M,s}$, there uniquely exists $\a_i\in\Sigma_M\cap\theta_s\Sigma_N$ such that $\fm_i\cap\theta_s\fn=\fg_{\a_i}$ for any $i\in A$. 

Let $\SO(2)^{\vee}$ be the unitary dual of $\SO(2)$. 
We fix an isomorphism $\BZ\cong\SO(2)^{\vee}; k\mapsto \x_k$ of topological groups so that $\sl_2(\BR)\cong \x_{-2}\oplus\x_0\oplus\x_2$ via the adjoint action.
For $1\le i\le r$, define a positive integer $k_i$ so that 
$\delta_i\cong\bigoplus_{\e\in\{\pm1\}}\bigoplus_{k\in\N}\x_{\e(k_i+2k)}$ as representations of $\SO(2)$.
Since the eigenvalue of the action of $X$ on $H_0((\fm_i\cap\theta_s\fn)_{\BC},(\delta_i)_{(K\cap M^s)\text{-fin}})\otimes \d_{P'_{M,  s}}^{-1/2}$ equals $\sum_{i\in A}k_i\a_i(X)/2$ (see the proof of \cite[Proof of Theorem 5.8]{ST}), 
any eigenvalue of the action of $X$ on \eqref{zero} is written as 
    \begin{align}\label{wt}
    \sum_{i\in A}\frac{k_i\a_i(X)}{2}+\l(X)
    +\sum_{1\le j\le k}\b_j(X)+\frac{(d\d_{P_s}-2d\d_{P_s^{\theta_s}})(X)}{2}
    \end{align}
where $\b_j$ denotes an $\fs^{\theta_s}$-weight appearing in $(\fg/(\fg^{\theta_s}+\fp))^{\vee}$ and $d$ denotes differential. 

From $k_i>0$ and \eqref{iii}, we see $k_i\a_i(X)>0$ for any $i\in A$. 
Since the real part of $\l$ is dominant, the condition \eqref{i} shows $\Re\l(X)\ge 0$.
Moreover, we see $\beta_j(X)>0$ by \eqref{ii}. 
Furthermore, Proposition \ref{prop}(2) shows 
    \begin{align}\label{d}
    (d\d_{P_s}-2d\d_{P_s^{\theta_s}})(X)
    & =\sum_{1\le i<j\le r}(-m_{ii}^+d_j+m_{ii}^-d_j)x_{ii}
    +(m_{jj}^+d_i-m_{jj}^-d_i)x_{jj}\\ \notag
    & =\sum_{1\le i<j\le r}d_id_j(-x_{ii}+x_{jj}) =
    \sum_{1\le i\le r}\left(\sum_{1\le j<i}d_j-\sum_{i<j\le r}d_j\right)d_ix_{ii} \\
    \notag
    & =\sum_{1\le i\le r} \{(d_{\le i}-d_i )+(d_{\le i}-[m]_2) \}d_ix_{ii}  \\ \notag
    & =\sum_{1\le i\le r} (2d_{\le i}-d_i -[m]_2)d_ix_{ii},
    \end{align}
where we used $d_{\le r}=[m]_2$ at the third equality. 
On the other hand,  we have
    \begin{align*}
    & \sum_{1\le i\le r-1}\left(d_{\le i}-[m]_2\right)d_{\le i}(x_{ii}-x_{i+1,i+1}) \\
    & \quad = (d_1-[m]_2)d_1x_{11} - \left(d_{\le r-1}-[m]_2\right)d_{\le r-1}x_{rr} \\
    &\qquad+ \sum_{2\leq i\leq r-1} \left\{\left(d_{\le i}-[m]_2\right)d_{\le i} 
    - \left(d_{\le i-1}-[m]_2\right)d_{\le i-1} \right\}x_{ii} \\
    & \quad = \sum_{1\le i\le r} (2d_{\le i}-d_i -[m]_2)d_ix_{ii}.
    \end{align*}
Hence we obtain
    \[
    (d\d_{P_s}-2d\d_{P_s^{\theta_s}})(X) 
    = \sum_{1\leq i\leq r-1}\left(d_{\le i}-[m]_2\right)d_{\leq i}(x_{ii}-x_{i+1,i+1}).
    \]

From $d_{\le i}\in\BZ$ and $[m]_2\in\{0,1\}$, we see 
$\left(d_{\le i}-[m]_2\right)d_{\le i}\ge 0$. 
Hence it follows from $x_{ii}> x_{i+1,i+1}$ that $(d\d_{P_s}-2d\d_{P_s^{\theta_s}})(X)\ge 0$ and the equality holds if and only if the condition \eqref{parity} is satisfied. 
Therefore,  if \eqref{wt} is zero,  then $A=\emptyset$,  $k=0$ and \eqref{parity} is satisfied.
This concludes the claim. 
\end{proof}

\begin{proof}[Proof of Theorem \ref{app}]
Given $s\in J_P(m)$, we have $M_s=M$ and there uniquely exists $\z_s\in S_r$ satisfying 
$\theta_s\diag(h_1,\ldots,h_r)=\diag(h_{\z_s(1)},\ldots, h_{\z_s(r)})$ for any $h_i\in\GL_{m_i}(D)$. In particular, we see $m_i=m_{\z_s(i)}$ for any $1\le i\le r$ and that $\z_s$ is an involution. 

Lemma \ref{lem} implies
    \[
    \dim\Hom_{H_m}(\pi,\triv_{H_m})\le \sum_{s\in J_P(m)}
    \dim\Hom_{M^{\theta_s}}(\delta_1\hat{\otimes}\cdots\hat{\otimes}\delta_r,\triv_{M^{\theta_s}}).
    \]
and the right hand side equals 
    \[
    \sum_{s\in J_P(m)}
    \left(\prod_{\z_s(i)=i}\dim\Hom_{\GL_{m_i}(D)^{\theta_s}}(\delta_i,\triv_{\GL_{m_i}(D)})\right)
    \left(\prod_{i<\z_s(i)}\dim\Hom_{\GL_{m_i}(D)}(\delta_i\hat{\otimes}\delta_{\z_s(i)},\triv_{\GL_{m_i}(D)})\right)
    \]
from the same argument as the proof of \cite[Theorem 5.13]{ST}.

Let us fix $1\le i\le r$. Recall $m_i\in\{1,2\}$. 
Assume $\z_s(i)=i$ and $m_i=1$. Then $m_{ii}=1$ and $\GL_{m_i}(D)^{\theta_s}=\GL_{m_i}(D)$. Hence 
    \begin{align}\label{mult}
    \dim\Hom_{\GL_{m_i}(D)^{\theta_s}}(\delta_i,\triv_{\GL_{m_i}(D)})\le 1
    \end{align}
and the equality holds if and only if $\delta_i\cong \triv_{\GL_{m_i}(D)}$. 

Assume $\z_s(i)=i$ and $n_i=2$, which implies $D=\BR$. 
Since $\theta$ satisfies \eqref{parity}, we have $m_{ii}^+=m_{ii}^-=1$ and $\GL_{m_i}(D)^{\theta_s}=H_2$. 
From the multiplicity one theorem,  we have \eqref{mult} also in this case and the equality holds if and only if $\delta_i$ is $H_2$-distinguished.

Assume $i<\z_s(i)$. By \cite[Lemma 3.9]{ST}, we have 
    \[
    \dim\Hom_{\GL_{m_i}(D)}(\delta_i\hat{\otimes}\delta_{\z_s(i)},\triv_{\GL_{m_i}(D)})\le 1
    \]
and the equality holds if and only if $\delta_{\z_s(i)}\cong\delta_i^{\vee}$. 

Put 
    \begin{align}\label{fT}
    \fT:=
    \{\z\in S_r | 
    \z^2=1, \text{ and \eqref{per} holds}
    \}.
    \end{align}
From the above arguments and the definition \eqref{fT} of $\fT$, we have 
    \[
    \dim\Hom_{H_m}(\pi,\triv_{H_m})\le \#\{s\in J_P(m) | \z_s\in \fT\}.
    \] 
Since the left hand side is positive, we have $\fT\neq \emptyset$ and the assertion follows.
\end{proof}

	\bibliographystyle{alphanum}
	\bibliography{references}
	
\end{document}